\documentclass{article}
\usepackage{ulem}
\usepackage{amssymb,amsbsy,amsmath,amsfonts,amssymb,amscd,amsthm}
\usepackage{xcolor}
\usepackage{bbm}
\usepackage{graphicx} 
\usepackage{multicol}
\graphicspath{{./}{immagini/}} 
\usepackage{tikz}
\usetikzlibrary{matrix}

\newcommand{\ov}{\overline}
\newcommand{\R}{\mathbb{R}}

\newcommand{\1}{\mathbbm{1}}

\newcommand{\mX}{\mathcal X}

\newcommand{\beq}{\begin{equation}}
\newcommand{\eeq}{\end{equation}}
\def\a{\alpha}

\def\d{\delta}

\def\g{\gamma}
\def\l{\lambda}

\def\m{\mu}

\def\s{\sigma}
\def\o{\omega}
\def\O{\Omega}

\def\gS{\Sigma}

\def\e{\varepsilon}

\def\half{\frac{1}{2}}

\newcommand{\cF}{{\cal F}}
\newcommand{\cG}{{\cal G}}
\newcommand{\cH}{{\cal H}}
\newcommand{\cI}{{\cal I}}

\newcommand{\cN}{{\cal N}}

\newcommand{\cL}{{\cal L}}

\newcommand{\cS}{{\cal S}}

\newcommand{\diver}{{\rm div}}

\newtheorem{theorem}{Theorem}[section]

\newtheorem{definition}[theorem]{Definition}
\newtheorem{proposition}[theorem]{Proposition}

\newtheorem{remark}[theorem]{Remark}
\numberwithin{equation}{section}
\begin{document}
\title{A system of of Hamilton-Jacobi equations characterizing     geodesic centroidal   tessellations}

\author{Fabio Camilli\footnotemark \and  Adriano Festa\footnotemark[2]}
\date{\today}
\maketitle
\footnotetext[1]{Dip. di Scienze di Base e Applicate per l'Ingegneria,   Sapienza  Universit{\`a}  di Roma, via Scarpa 16, 00161 Roma  ({\tt  fabio.camilli@uniroma1.it})}
\footnotetext[2]{DISMA - Dipartimento di Scienze Matematiche ``Giuseppe Luigi Lagrange", Corso Duca degli Abruzzi, 24, 10129 Torino ({\tt  adriano.festa@polito.it})}
\footnotetext[3]{The present research was partially supported by MIUR grant ``Dipartimenti Eccellenza 2018-2022" CUP: E11G18000350001, DISMA, Politecnico di Torino.}
\begin{abstract}  We introduce a  class of systems of Hamilton-Jacobi equations   characterizing  geodesic centroidal tessellations, i.e. tessellations  of domains with respect to geodesic distances where generators and centroids coincide. Typical examples are given by  geodesic centroidal Voronoi tessellations and geodesic centroidal power diagrams. An appropriate version of the Fast Marching method on unstructured grids  allows computing the solution of the Hamilton-Jacobi system and therefore   the associated tessellations. We propose various numerical examples to illustrate the features of the technique. 
  \end{abstract}

\noindent
{\footnotesize \textbf{AMS-Subject Classification:} 65K10, 49M05, 65D99, 35F21, 49N70}.\\
{\footnotesize \textbf{Keywords:} geodesic distance; Voronoi tessellation;  $K$-means; power diagram;  Hamilton-Jacobi equation; Mean Field Games; Fast Marching method}.

\section{Introduction}
A partition, or tessellation, of a set $\Omega$   is a collection of mutually disjoint subsets $\Omega_k\subset \Omega$, $k=1,\dots,K$, such that $\cup_{k=1}^K   \Omega_k=\Omega$. A classical model is the Voronoi tessellation and, in this case, the sets $\Omega_k$ are called Voronoi diagrams.  Tessellations and other similar families of geometric objects arise in several applications, ranging from graphic design, astronomy, clustering, geometric modelling,  data analysis, resource optimization, quadrature formulas, and discrete integration, sensor networks, numerical methods for partial differential equations (see \cite{akl,okabe}).\\
Partitions and tessellations are frequently associated with objective functionals, defining desired additional properties to be satisfied. A well-known example is the $K$-means problem in cluster analysis, which aims to subdivide a data set into $K$ clusters such that each data point belongs to the cluster with the nearest cluster center. Minima of the $K$-means functional result in a partitioning of the data  in centroidal Voronoi diagrams, i.e., Voronoi diagrams for which generators and centroids coincide (see \cite{dfg_rev}).  In other applications, the size of the cells is prescribed (capacity-constrained problem), and the  partition of $\Omega$  is given by another generalization of Voronoi diagrams, called power diagrams (\cite{aha,bourne}).\\
Algorithms for the computation of centroidal Voronoi tessellations in the Euclidean case, such as the Lloyd's algorithm, exploit  geometric properties of the problem to rapidly converge to a solution. 
The case of  geodesic Voronoi tessellation, i.e. tessellation with respect to a general convex metric,   presents additional difficulties both in the computation of Voronoi diagrams and in that of the corresponding centroids.\\
In this work, we introduce a  PDE method for  the computation of the geodesic Voronoi tessellation.  Given a    density function $\rho$ supported in a bounded set $\Omega$, representing the distribution of a data set, the aim is to subdivide the point in $K$ clusters defined by  a convex metric $d_C$.  As a prototipe of the approach, shown in its simplest form, we introduce a system of first-order Hamilton-Jacobi (HJ in short) equations that, for the Euclidean distance, reads as 
\begin{equation}\label{HJ}
\left\{
\begin{array}{ll}
|Du_k| =   1,\quad  x\in\R^d,\, k=1,\dots,K \\[6pt]
u_k( \m_k)=0,\\[6pt]
S_u^k=  \{x\in\R^d:\,u_k(x)=\min_{j=1,\dots,K} u_{j}(x)\}, \\[6pt]
\mu_{k}=\frac{\int_{S_u^k }x \rho(x)dx}{\int_{S_u^k}  (x)\rho(x)dx}.
\end{array}
\right.
\end{equation}
We show that the family $\{S_u^k\}_{k=1}^K$  defined  by \eqref{HJ} corresponds to a critical point of the $K$-means functional, hence to a centroidal Voronoi tessellation of $\Omega$ with centroids $\m_k$; vice versa, to each critical point  of the functional corresponds to a solution $u=(u_1,\dots,u_K)$ of the previous system. Moreover, a  system of HJ equations similar to \eqref{HJ} provides a way to compute the optimal weights for the capacity-constrained problem, which aims to find a geodesic centroidal tessellation of the domain with regions of a given area. This problem arises in several applications in economy, and it is connected with the so-called semi-discrete Optimal Transport problem (\cite{levy,m}).\\
It is well known that the hard clustering $K$-means problem can be seen as the limit of the soft clustering Gaussian mixture model  when the variance parameter   goes to $0$ (see \cite{bishop}).
Relying on this observation, we provide an interpretation of   system \eqref{HJ}   as the vanishing viscosity limit of a multi-population Mean Field Games (MFG in short) system introduced in \cite{accd} to characterize the   parameters of a  mixture model maximizing a log-likelihood functional.\\
To solve the system \eqref{HJ} we consider an iterative method similar to the LLoyd's algorithm. At each step, given the generators of the tessellation computed in the previous step, we compute  the Voronoi diagrams  solving the HJ equation via a Fast Marching technique. Then, we compute the new generators and we iterate. As we discuss later, smart management of the data and the use of acceleration techniques may considerably speed up the process.\\
PDE theory is  a robust framework to solve classic (and less traditional) tessellation problems. The main advantage of this approach is the high adaptability of the framework to specific variations of the problem (presence of constraints, non-conventional distance functions, etc.). This increased adaptability comes with a precise cost: a PDE approach is more computationally demanding than other methods available in the literature. However, the recent developments of numerical methods for nonlinear PDEs, and the increment of the accessibility to more powerful computational resources at any level, make these techniques progressively more appealing in many applicative contexts \cite{Festa16, Saluzzi19,Kalise18}. \par
The paper is organized as follows. In Section \ref{sec:K_means}, we introduce a  HJ system approach to the hard-clustering problem  and   geodesic centroidal Voronoi tessellations. In Section \ref{sec:power_diagrams}, we consider a system of HJ equations to characterize centroidal power diagrams, a generalization of centroidal Voronoi tessellations where the measure of the cells is prescribed.
In Section \ref{sec:MFG_inter}, we provide an interpretation of the HJ system in terms of MFG theory. In Section \ref{sec:numerics}, we discuss the numerical approximation of the HJ systems introduced in the previous sections and we provide several examples.

\section{Geodesic Voronoi tessellations  and HJ equations}\label{sec:K_means}
In this section, we introduce a class of geodesic distance, the corresponding $K$-means problem and its characterization via a system of Hamilton-Jacobi equations. Consider a set-valued map $x\mapsto C(x)\subset \R^d$ and assume that
 \begin{itemize}
 	\item[(i)]   for each $x\in\R^d$,   $C(x)$ is a compact, convex set and $0\in C(x)$;
 	\item[(ii)] there exists $L>0$ such that  $d_{\cH}(C(x),C(y))\le L|x-y|$,  for all $x,y\in\R^d$;	
 	\item[(iii)] there exists $\d>0$ such $B(0,\d)\subset C(x)$ for any $x\in\R^d$,
 \end{itemize}
where $d_{\cH}$ denotes the Hausdorff distance.
 For $x,y\in\R^d$, let $\cF_{x,y}$ be the set of all the trajectories $X(\cdot)$ defined by the differential inclusion
 \[\dot X(t)\in C(X(t)), \, X(0)=x,\,X(T)=y,\]
 for some $T=T(X(\cdot))>0$. Note that, because  of the assumptions on the map $C(x)$,  $\cF_{x,y}$ is not empty.
 The   function $d_C:\R^d\times\R^d\to\R$, defined  by
 \begin{equation}\label{eq:gendist_dist}
 	d_C(x,y)=\inf_{\cF_{x,y}} T(X(\cdot)),
 \end{equation}
 is a distance function, equivalent to the Euclidean distance (see \cite{cdi}). Some   examples of distance $d_C$ are provided at the end of this section, see Remark \ref{rem:ex_dist}.\\
We introduce the $K$-means problem for the geodesic distance $d_C$. Let $\Omega$ be a bounded subset of $\R^d$ and  $\rho$  a density function supported in $\Omega$, i.e. $\rho\ge 0$ and $\int_{\Omega} \rho dx=1$,    representing the distribution of the points of a  given data set $\mX$.  The $K$-means problem for the distance  $d_C$ aims to minimize the functional 
\begin{equation}\label{eq:gendist_functional}
	\begin{split}
	&\cI_{C}(y_1,\dots,y_k)=\sum_{k=1}^K\int_{V(y_k)}d_C(x,y_k)^2\rho(x)dx,\\
	&\text{where}\quad V(y_k)=\{x\in\R^d:d_C( y_k,x)=\min_{j=1,\dots,K}d_C( y_j,x)\}.
	\end{split}
\end{equation} 
A  minimum of the functional $\cI_{C}$  provides a clusterization of the data set, i.e., a repartition of   $\mX$ into $K$ disjoint clusters $V(y_k)$ such that each data point belongs to the cluster with the smaller distance from centroid  $y_k$.  This property can be expressed in the elegant terminology of the \textit{geodesic centroidal Voronoi tessellations} (see \cite{dfg_rev,dfg_sinum,lylh}).
Given  a set of generators  $\{y_k\}_{k=1}^K$, $y_k\in\ov\Omega$, we define a geodesic Voronoi tessellation   of $\Omega$   as the union of the geodesic Voronoi diagrams
\begin{equation}\label{eq:gendist_GVD}
	V(y_k)=\{x\in\Omega: \,d_C(x,y_k)=\min_{j=1,\dots,K} d_C(x,y_j)\}
\end{equation}
(a point of $V(y_k)\cap V(y_j)$ is assigned to the diagram with the smaller index).
\begin{definition}
	A geodesic Voronoi tessellation $\{V(y_k)\}_{k=1}^K$ of $\Omega$ is said to be a geodesic centroidal Voronoi tessellation (GCVT in short)  if, for each  $k=1,\dots,K$, the generator $y_k$ of $V(y_k)$ coincides with the centroid of $V(y_k)$, i.e.
	\begin{equation}\label{eq:gcvt_def}	
		\int_{V(y_k)}\rho(x)d_C(y_k,x)dx=\min_{z\in V(y_k)}\int_{V(y_k)}\rho(x)d_C(z,x)dx.	
	\end{equation}
\end{definition}
\begin{remark}
	If    $C(x)=B(0,1)$ for each $x\in\R^d$, then  $d_C$ coincides with   the Euclidean distance and \eqref{eq:gendist_functional} is the classical $K$-means problem (see \cite{dfg_sinum}). In this case, $\{V(y_k)\}_k$ is called a centroidal Voronoi tessellation (CVT in short) and the 
	centroids are given by
	\begin{equation}\label{centroids}
		y_k=\frac{\int_{V(y_k)} s\rho(s)ds}{\int_{V(y_k)}  \rho(s)ds}.
	\end{equation}
\end{remark}
Since $\Omega$ is bounded and   $\cI_{C}$ is continuous, a global minimum   of the functional \eqref{eq:gendist_functional} exists; but, since $\cI_{C}$ is in general non convex, local minimums  may also exist. In \cite[Thereom 1]{lylh}, it is proved that the previous functional is continuous and
\begin{equation}\label{eq:intro_equiv_CVT}
	\text{critical points of   $\cI_{C}$    correspond to   GCVTs of $\Omega$.}
\end{equation}  
Critical points of $\cI_{C}$ can be computed via the Lloyd algorithm, a simple two steps iterative procedure. Starting from an arbitrary initial set of generators, at each iteration the following two steps are performed
\begin{itemize}
	\item Given the set of generator $\{y_i\}_{i=1}^K$ at the previous step, construct the Geodesic Voronoi tessellation $\{V(y_i)\}_{i=1}^K$ as in \eqref{eq:gendist_GVD};
	\item take the centroids of  $\{V(y_i)\}_{i=1}^K$ as the new set of generators and iterate.
\end{itemize}
The procedure is repeated until an appropriate stopping criterion is met. At each iteration, the objective function $\cI_{C}$ decreases and the algorithm converges to a (local) minimum of  \eqref{eq:gendist_functional} (see \cite[Theorem 2.3]{dfg_sinum} in the Euclidean case and \cite{lylh} in the general case).\par 
In order to introduce a PDE characterization of GCVT, we associate to the distance $d_C$ a Hamiltonian  $H:\R^d\times\R^d\to \R$  defined  as the support function  of the convex set $C$, i.e. 
\begin{equation*}
	H(x,p)=\sup_{q\in C(x)}p\cdot q.
\end{equation*}
Then  $H:\R^d\times\R^d\to\R$ is a continuous function and satisfies the following properties
\begin{itemize}
\item[(i)]	$H(x,0)=0$, $H(x,p)\ge 0$ for $p\in\R^d$;
\item[(ii)]	$H(x,p)$ is convex and positive homogeneous in $p$,
	i.e. for  $\l>0$, $H(x,\l p)=\l H(x,p)$;
\item[(iii)]	$ |H(x,p)-H(y,p)|\le L|x-y|(1+|p|)$ for $x,y\in\R^d$.
\end{itemize}
Moreover, for any $y\in\R^d$, the function $u:\R^d\to \R$, defined by $u(x)=d_C(y,x)$, is the unique viscosity solution (see \cite{bcd} for the definition) of the problem
\begin{equation}\label{eq:gendist_HJ}
	\left\{
	\begin{array}{ll}
		H(x,Du)= 1 ,\quad  x\in\R^d, \\[6pt]
		u(y)=0.\\[6pt]
	\end{array}
	\right.
\end{equation}
 \\
We characterize  GCVTs of $\Omega$ via  the following system of HJ equations
\begin{equation}\label{eq:gendist_MFG}
	\left\{
	\begin{array}{ll}
		H(x,Du_k)=1 , \\[6pt]
		u_k( \m_k)=0,\\[6pt]
		S_u^k=  \{x\in\R^d:\,u_k(x)=\min_{j=1,\dots,K} u_{j}(x)\},  \\[6pt]
		\int_{S^k_u}\rho(x)u_k(x)dx=\min\{\int_{S^k_u}\rho(x)u_y(x)dx: \text{$u_y$ solution of \eqref{eq:gendist_HJ} with $y\in S^k_u$}\}.
	\end{array}
	\right.
\end{equation}
for  $k=1,\dots, K$.\\ 
Recall that   the unique solution of \eqref{eq:gendist_HJ} is given by  $u(x)=d_C( y,x)$, hence 
$u_k(x)=d_C(\mu_k,x)$. Furthermore, the last condition in \eqref{eq:gendist_MFG}, see also \eqref{eq:gcvt_def}, implies  that the points $\mu_k$ are the centroids of the sets $ S^k_u$ with respect to the metric $d_C$. On the other hand,   the HJ equations   are coupled via the points $\mu_1,\dots,\mu_k$ which are  the centroids of the  sets   $S^k_u$, $k=1,\dots,K$ and therefore they are unknown. Indeed, the true unknowns in system \eqref{eq:gendist_MFG} are the  points $\mu_k$, $k=1,\dots,K$, since   they determine the functions $u_k$ as viscosity solutions of the corresponding HJ equations and consequently the diagrams $S^k_u$. \\
We now show that the  previous system characterizes   critical points of the functional \eqref{eq:gendist_functional} or, equivalently, GCVTs of the set $\Omega$.
\begin{proposition}\label{prop:gendist_equivalence}
	The following conditions are equivalent:
	\begin{itemize}
		\item[(i)] Let $(y_1,\dots,y_K)$ be a critical point of the functional $\cI_{C}$ in \eqref{eq:gendist_functional} with geodesic Voronoi diagrams $V(y_k)$. Then, there exists a solution of \eqref{eq:gendist_MFG} such that $\m_k=y_k$ and $S_u^k=V(y_k)$.
		\item[(ii)] Given a solution $u=(u_1,\dots,u_K)$ of \eqref{eq:gendist_MFG}, then  $(\m_1,\dots,\m_K)$ is a critical point of $\cI_{C}$ with geodesic Voronoi diagrams $V(y_k)=S_u^k$.
	\end{itemize}
\end{proposition}
\begin{proof}
	Assume that $(y_1,\dots,y_K)$ is a critical point of the functional $\cI_{C}$, hence  $V(y_k)$ defined as in \eqref{eq:gendist_GVD} is a GCVT and
	\begin{equation}\label{eq:intro_proof1}
		\int_{V(y_k)}\rho(x)d_C(y_k,x)dx=\min_{z\in V(y_k)}\int_{V(y_k)}\rho(x)d_C(z,x)dx, \qquad \forall k=1, \dots, K.
	\end{equation}
	 Define  $u=(u_1,\dots,u_K)$, $\m=(\mu_1,\dots, \mu_k)$  by
	\begin{equation}\label{eq:intro_proof2}
		u_{k}(x)=  d_C(y_k,x), \qquad \mu_k=y_k, \quad k=1,\dots,K.
	\end{equation}
	Then $u=(u_1,\dots,u_K)$ is a solution of the HJ equations in \eqref{eq:gendist_MFG} with $\m_k=y_k$. Moreover, by \eqref{eq:gendist_functional}, we have that $S_u^k=V(y_k)$
	and therefore \eqref{eq:intro_proof1}	 is equivalent to
\begin{align*}
	 \int_{S^k_u}\rho(x)u_k(x)dx=\min\{\int_{S^k_u}\rho(x)u_z(x)dx: \text{$u_z$ solution of \eqref{eq:gendist_HJ} with $z\in S^k_u$}\}.
\end{align*}
	We conclude that   $u=(u_1,\dots,u_K)$  and $\m=(\mu_1,\dots, \mu_k)$ in \eqref{eq:intro_proof2} give  a solution of \eqref{eq:gendist_MFG}.\par 
	Now assume that $u=(u_1,\dots,u_K)$, $\m=(\mu_1,\dots, \mu_k)$ is a solution of \eqref{eq:gendist_MFG} and set $y_k=\mu_k$, $k=1,\dots, K$. Then, defined $V(y_k)$ as in \eqref{eq:gendist_functional}, we have
	$V(y_k)=S_u^k$. Moreover, taking into account that $u_{k}(x)= d_C(\mu_k,x)$ and $\mu_k$ are characterized by
	$$\int_{S^k_u}\rho(x)u_k(x)dx=\min\{\int_{S^k_u}\rho(x)u_y(x)dx: \text{$u_y$ solution of \eqref{eq:gendist_HJ} with $y\in S^k_u$}\}
	$$
	 we also have that $y_k$ satisfies \eqref{eq:gcvt_def}. Therefore $V_k$, $k=1,\dots, K$, is a GCVT and, by \eqref{eq:intro_equiv_CVT}, $y_k$, $k=1,\dots, K$, a minimum of $\cI_{C}$.
\end{proof}
The previous result can be restated in the terminology of the Voronoi tessellation, saying that a solution of the   system \eqref{eq:gendist_MFG} determine a GCVT  and vice versa.
We have the following existence result for \eqref{eq:gendist_MFG}.
\begin{theorem}
	Let $\rho$ be a positive and smooth density function defined on a smooth bounded set $\Omega$. Then, there exists a solution to \eqref{eq:gendist_MFG}. 
	Moreover, any limit point   of the Lloyd algorithm  corresponds to  a solution of the HJ system.
\end{theorem}
\begin{proof}
	The  first assertion is consequence of  existence of critical  points of the functional $\cI_{C}$ and  the equivalence result provided by Prop. \ref{prop:gendist_equivalence}.
	The second part of the statement follows from the  convergence of the Lloyd algorithm
	and standard stability results in viscosity solutions theory.
\end{proof}

\begin{remark}\label{rem:ex_dist}
We give some examples of geodesic distance $d_C$ and the corresponding  Hamiltonian $H$.
	\begin{enumerate}
		\item if $C(x)=\{p\in\R^d:\,\|p\|_s=(\sum_{i=1}^d|p_i|^s)^{1/s}\le 1\}$ for $s>1$, then $ d_C$ is the  Minkowski distance $d_C(x,y)= \|x-y\|_s$ and $H(x,p)=|p|^2/\|p\|_s$;
		\item if $C(x)=a(x)B(0,1)$, where $a(x)\ge \d>0$, then $H(x,p)=a(x)|p|$. In particular, the Euclidean case corresponds to $a(x)\equiv 1$;
		\item if $C(x)=A(x)^{\half}B(0,1)$, where $A$ is a positive definite matrix such that $A(x)\xi\cdot\xi\ge \d>0$ 
		for $\xi\in\R^d$, then $d_C$ is the Riemannian distance induced by the matrix $A$ on $\R^d$ and $H(x,p)=\sqrt{A(x)p\cdot p}$.
	
	\end{enumerate}
	Moreover it is possible to consider the distance function corresponding to a Hamiltonian $H$ defined by 
	\[H(x,p)=\max\{H_1(x,p),\cdots,H_N(x,p)\},\]
	where $H_n$, $n=1,\dots,N$ are Hamiltonians of the types above.
	
	\begin{figure}[!t]
	\begin{center}
		\includegraphics[width=.4\textwidth]{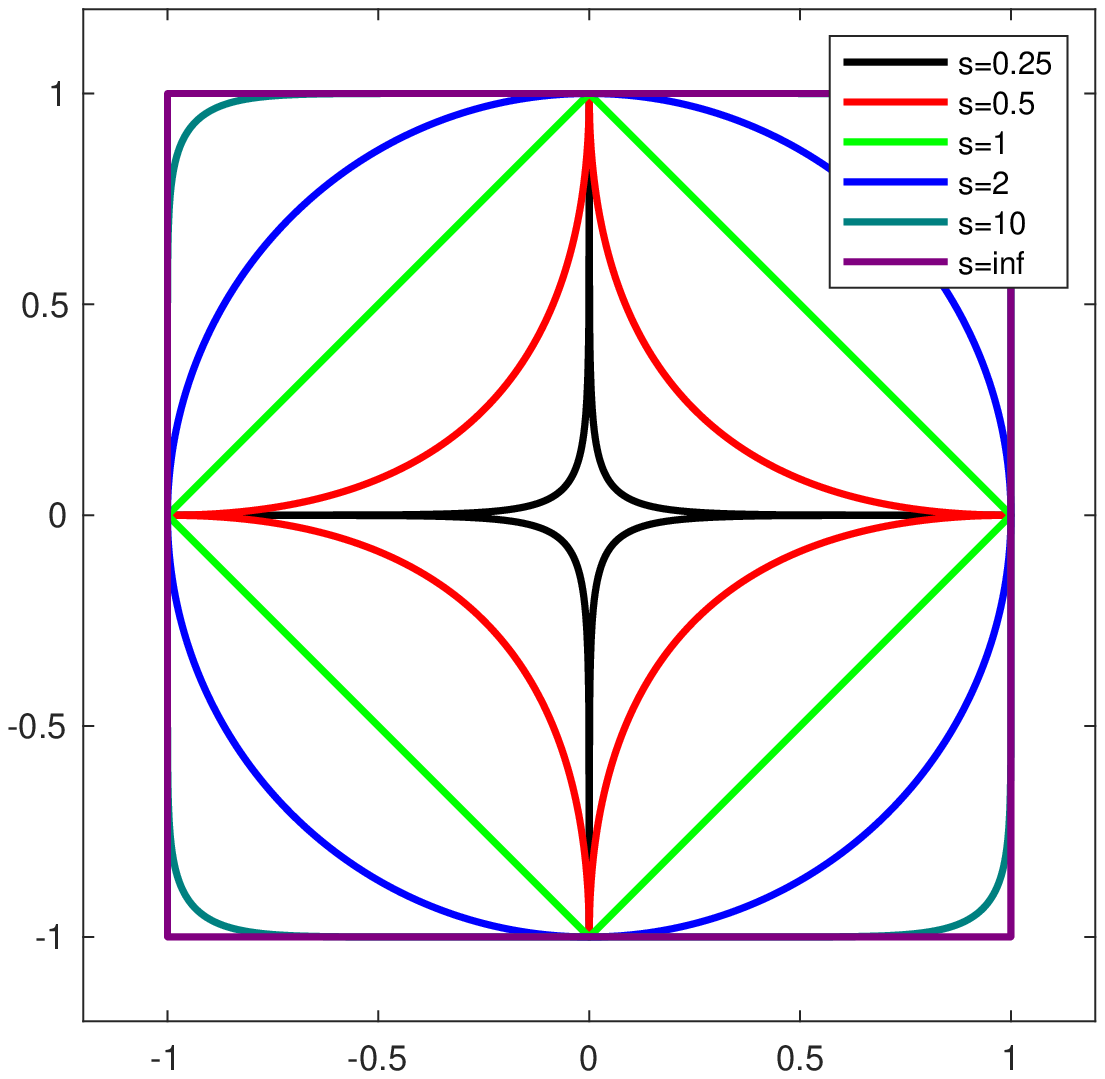}
		\includegraphics[width=.42\textwidth]{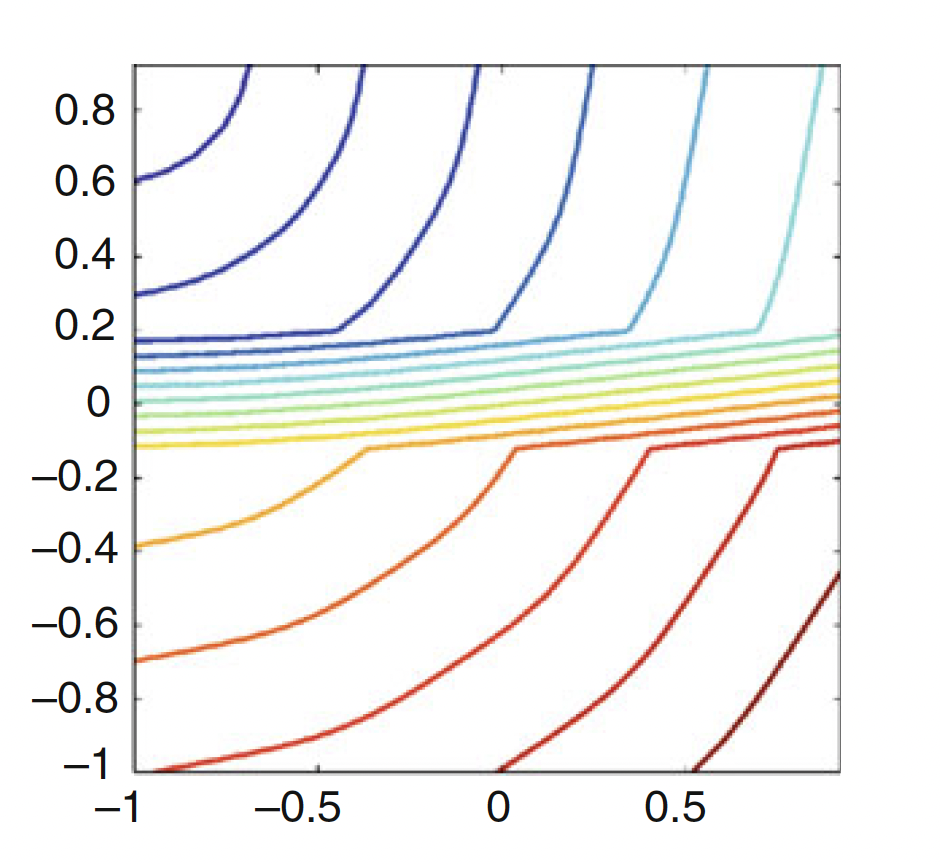}
	\end{center}
	\caption{Unitary balls $B(0,1)$ in the Minkowski  distance for various values of $s$ (left) and a Riemann distance induced by $A(x)=(1,0; 0,(1-0.8\, \chi_{|y|<0.2})^2)$.}\label{figmi}
\end{figure}

\end{remark}

\section{A system of HJ equations for geodesic centroidal power diagrams}\label{sec:power_diagrams}
In this section, we consider a  generalization of  centroidal Voronoi diagrams, called  centroidal power diagrams. We first introduce the definition of power diagrams, or weighted Voronoi diagrams, and then describe centroidal power diagrams and a system of HJ equations that can be used to compute them.\\
Consider the distance $d_C$ defined as in \eqref{eq:gendist_dist}. Given a set  of $K$ distinct points  $\{y_i\}_{i=1}^K$ in $\Omega$ and $K$ real numbers   $\{w_i\}_{i=1}^K$, the geodesic power diagrams generated by the couples $(y_i,w_i)$  are defined by
\begin{equation} \label{eq:pw_def}
V(y_k,w_k)=\{x\in\Omega: \,d_C(x,y_k)-w_k=\min_{j=1,\dots,K}(d_C(x,y_j)-w_j)\}.
\end{equation}
 As Voronoi diagrams, power diagrams provide a tessellation of the domain $\O$, i.e. ${V(y_i,w_i)}^{\mathrm{o}}\cap {V(y_j,w_j)}^{\mathrm{o}}=\emptyset$ for $i\neq j$ and $\cup_{i=1}^K V(y_i,w_i) =\ov \Omega$. Note that, whereas Voronoi diagrams are always non empty,      some of the power diagrams  may be empty and the corresponding generators belong to another diagram. Power diagrams reduce to Voronoi diagrams if the weights $w_i$ coincide, but they   have an additional tuning parameter, the weights vector
 $w=(w_1,\dots,w_k)$, which allows to impose additional constraints on the resulting tessellation.\\
A typical application of power diagrams is the problem of partioning a given set in a capacity constrained manner (see \cite{aha}).  Given  a density function $\rho$ supported in $\Omega$ and $K$ distinct points $\{y_i\}_{i=1}^K$ in $\Omega$, consider the measure $\pi(dx)=\rho(x)dx$   and, to each point $y_i$,  associate a cost $c_i>0$ with the property that $\sum_{i=1}^Kc_i=\pi(\Omega)$. For a partition of $\Omega$ in a family of $K$ subsets $R_i$, define the cost of each subset as $\pi(R_i)=:\int_{R_i}d_C(x,y_i)\,\pi(dx)$.
The aim is to find a partition $\{R_i\}_{i=1}^K$  of $\Omega$ such that    the total cost
\begin{equation*}
 Q(R_1,\dots,R_K)=\sum_{i=1}^K\int_{R_i}d_C(x,y_i)\pi(dx)
 \end{equation*}
is minimized under the constraint $\pi(R_i)=c_i$. In \cite[Theorem 1]{levy_et_altri}, it is shown that the minimum of the previous functional exists and it is reached by a geodesic power diagram  generated by the couples $(y_k,w_k)$, $k=1,\dots,K$,  where the unknown weights $w_k$ can be found by maximizing    the concave functional 
\begin{equation}\label{eq:pw_reg_func}
\cF(w_1,\dots,w_k)=\sum_{i=1}^K\int_{V(y_i,w_i)}d_C(x,y_k)\rho(x)dx-\sum_{i=1}^kw_i( \pi(V(y_i,w_i))-c_i).
\end{equation}
The gradient of $\cF$ is given by
\begin{equation*}
\frac{\partial \cF}{\partial w_i}=c_i- \pi(V(y_i,w_i))
\end{equation*}
and, if $(w_1,\dots,w_k)$ is a critical point of $\cF$, then the power diagram  generated by the couples $(y_i,w_i)$ satisfies the capacity constraint $\pi(V(y_i,w_i))=c_i$.  The previous optimization problem is also connected with the  semi-discrete optimal mass transport problem, i.e.  optimal transport of a continuous measure $\pi$ on a discrete measure $\nu=\sum_{i=1}^K c_i\d_{y_i}$ (see \cite{levy,m}). Algorithms to compute critical points of \eqref{eq:pw_reg_func} are described in \cite{degoes,levy,m}.\\

We   consider geodesic centroidal  power diagrams, i.e. geodesic power diagram for which   generators coincide with the corresponding centroids. Indeed, it   has been observed that the use of centroidal power diagrams in the capacity constrained partionining problem avoid generating irregular or elongated cells  (see \cite{bourne,levy_et_altri}). 
\begin{definition}
	A geodesic power diagram tessellation $\{V(y_i,w_i)\}_{i=1}^K$ of $\Omega$ is said to be a geodesic centroidal power diagram  tessellation   if, for each  $i=1,\dots,K$, the generator $y_i$ of $V(y_i,w_i)$ coincides  with the centroid of $V(y_i,w_i)$, i.e.
\[\int_{V(y_k,w_k)}\rho(x)d_C(y_k,x)dx=\min_{z\in V(y_k,w_k)}\int_{V(y_k,w_k)}\rho(x)d_C(z,x)dx.\]
\end{definition}
In \cite{levy_et_altri}, geodesic  centroidal power diagrams satisfying the capacity constraints  $\pi(V(y_k,w_k))=c_k$ are characterized as a saddle point of the functional
\begin{align*}
	\cG(y_1,\dots,y_k,w_1,\dots,w_k)=&\sum_{i=1}^K\int_{V(y_i,w_i)}d_C(x,y_i)\rho(x)dx\\
	&-\sum_{i=1}^kw_i( \pi(V(y_i,w_i))-c_i).
\end{align*}
Note that the previous functional is  similar to one defined in \eqref{eq:pw_reg_func}, but   it depends also on the generators $(y_1,\dots,y_k)$. For $(y_1,\dots,y_k)$ fixed, $\cG$ is concave with respect to
$w=(w_1,\dots,w_K)$ and therefore it admits a maximizer which determine a power diagram 
$\{V(y_i,w_i)\}_{i=1}^K$. For $(w_1,\dots,w_k)$ realizing the capacity constraints $\pi(V(y_i,w_i))=c_i$, $\cG$ coincides with the functional $\cI_{C}$ in \eqref{eq:gendist_functional}, hence it is minimized by the centroids of sets 
$\{V(y_i,w_i)\}_{i=1}^K$.  \\
We propose the following HJ system for the characterization of the saddle points of $\cG$
\begin{equation}\label{eq:pw_MFG_CPD}
	\left\{
	\begin{array}{ll}
		H(x,Du_k)=1 ,\qquad x\in\Omega, \\[6pt]
		u_k( \mu_k)=-\omega_k,\\[6pt]
		S_u^k=  \{x\in\R^d:\,u_k(x)=\min_{j=1,\dots,K} u_{j}(x)\}\\[6pt]
		\int_{S^k_u}\rho(x)u_k(x)dx=\min\{\int_{S^k_u}\rho(x)u_y(x)dx:\, \\
		\hspace{4.5cm}\text{$u_y$ solution of \eqref{eq:gendist_HJ} with $y\in S^k_u$}\},\\[6pt]
		\pi(S_u^k)=c_k.
	\end{array}
	\right.
\end{equation}
The previous system  depends on the $2K$ parameters $(\mu_k,\omega_k)$. A solution of
\begin{equation*}
	\left\{
	\begin{array}{ll}
	H(x,Du) =1,  \\[6pt]
		u( y)=-\omega,
\end{array}
\right.
\end{equation*}
is given by $u_y(x)=-\omega+d_C(y,x)$. Hence, if there
exists a solution $u=(u_1,\dots,u_k)$ to  \eqref{eq:pw_MFG_CPD}, then $u_k(x)=-\omega_k+d_C(\mu_k,x)$. Moreover   
$$S_u^k= \left\{x\in\R^d:-\o_k+  d_C(\mu_k,x)=\min_{j=1,\dots,K}\{-\o_j+  d_C(\mu_j,x)\}\right\}$$
and  $\mu_k$ is the centroid of $S_u^k$. It follows that the set $S_u^k$ coincides 
 $V(y_k,w_k)$  defined  in \eqref{eq:pw_def}, and $\pi(S_u^k)=c_k$.
We conclude that a solution of \eqref{eq:pw_MFG_CPD} gives a centroidal power diagram  $\{S_u^k\}_{k=1}^K$ of $\O$
realizing the capacity constraint.\\

\section{A Mean Field Games  interpretation of the HJ system}\label{sec:MFG_inter}
In this section, we establish a link between the Hamilton-Jacobi system   introduced in Section \ref{sec:K_means}  and the theory of Mean Field Games (see \cite{cd,ll,gs} for an introduction).  We show that the HJ system   \eqref{eq:gendist_MFG}    can be obtained in the vanishing viscosity  limit of a second order multi-population MFG system characterizing the extremes of a maximal likelihood functional.\\
Finite mixture models, given by a convex combination of probability density functions, are a powerful tool for statistical modeling of data, with applications to pattern recognition, computer vision, signal and image analysis, machine learning, etc (see \cite{bishop}).
Consider  a Gaussian mixture model 
\begin{equation}\label{eq:intro_mix}
	m(x)=\sum_{k=1}^K \a_k \cN(x;\m_k,\gS_k), \quad \text{with $\a_k\in (0,1)$, \; $\sum_{k=1}^K\a_k=1$},
\end{equation}
where  $\m_k$ and $\gS_k$ denote  mean and   covariance matrix of the Gaussian distribution $\cN(x;\m_k,\gS_k)$. 
The aim is to determine  the parameters $\a=(\a_1,\dots,\a_K)$, $\m=(\m_1,\dots,\m_K)$, $\gS=(\gS_1,\dots,\gS_K)$ of the mixture \eqref{eq:intro_mix}    in such a way that they optimally fit a given data set $\mX$ described by the density function $\rho$. This can obtained by maximizing the  log-likelihood functional 
\begin{equation}\label{eq:intro_log}
	\cL(\a,\m,\gS;\mX)=\int_{\R^d}\sum_{k=1}^{K}\g_k(x)\{\ln(\a_k)+
	\ln(\cN(x;\m_k,\gS_k))\} \rho(x)dx,
\end{equation}
where 
\begin{equation*}
	\g_k(x)= \dfrac{\a_k \mathcal{N}(x; \mu_k, \Sigma_k)}{\sum_{j=1}^{K}\a_j \mathcal{N}(x; \mu_j, \Sigma_j)}
\end{equation*}
are the responsibilities, or posterior probabilities (see \cite[Cap. 7]{bishop} for more details). \\
In \cite{accd}, we propose an alternative approach to parameter optimization for mixture models based on the MFG theory.  It can shown   that  the critical points of the log-likelihood functional \eqref{eq:intro_log} can be characterized by means of the multi-population MFG system
\begin{equation}\label{eq:intro_MFG}
	\left\{
	\begin{array}{ll}
		-\e  \Delta u_{k,\e} + \half |D u_{k,\e} |^2+\l_{k,\e}= \frac{\e^2}{2} (x-\mu_{k,\e})^t(\Sigma_{k,\e}^{-1})^t\Sigma_{k,\e}^{-1}(x-\mu_{k,\e}) ,\quad& x\in\R^d, \\[8pt]
		\e    \Delta m_{k,\e} + \diver(m_{k,\e} D u_{k,\e}  )=0,&x\in\R^d, \\[8pt]
		\a_{k,\e}=\int_{\R^d }\g_{k,\e}(x) \rho(x)dx,\\[8pt] 
		m_{k,\e}\ge 0,\,\int_{\R^d} m_{k,\e} dx=1,    u_{k,\e}( \m_{k,\e})=0,
	\end{array}
	\right.
\end{equation}
for   $k=1, \dots, K$, where 
\begin{align}
	&\g_{k,\e}(x)=\dfrac{\a_{k,\e} m_{k,\e}(x)}{\sum_{j=1}^K\a_{j,\e} m_{j,\e}(x)},\nonumber\\
	&\mu_{k,\e}=\frac{\int_{\R^d}x\gamma_{k,\e}(x)\rho(x)dx}{\int_{\R^d} \gamma_{k,\e}(x)\rho(x)dx},\label{eq:intro_resp}\\[8pt]
	&\Sigma_{k,\e}=\frac{\int_{\R^d}(x-\mu_{k,\e})(x-\mu_{k,\e})^t\gamma_{k,\e}(x)\rho(x)dx}{\int_{\R^d} \gamma_{k,\e}(x)\rho(x)dx} \nonumber
\end{align}  
are unknown variables which depend on the solution of the system \eqref{eq:intro_MFG}. 
More precisely, a solution of \eqref{eq:intro_MFG} is given by a family of quadruples  $(u_{k,\e}, \lambda_{k,\e}, m_{k,\e}, \a_{k,\e})$, $k=1,\dots,K$, with 
\begin{equation}\label{eq:intro_sol}
	\begin{split}
		&u_{k,\e}(x)=\frac{\e}{2}(x-\m_{k,\e})^t\gS^{-1}_{k,\e}(x-\m_{k,\e}),\quad\lambda_{k,\e}=\e^2 \text{Tr}(\gS_{k,\e}^{-1}),\\
		&m_{k,\e}(x)=\cN(x;\m_{k,\e},\gS_{k,\e})=C_ke^{-\frac{u_k(x)}{\e}},\\
		&\a_{k,\e}=\int_{\R^d}\gamma_{k,\e}(x)\rho(x)dx,
	\end{split}
\end{equation} 
and the corresponding  parameters  $(\a_{k,\e}, \m_{k,\e}, \gS_{k,\e})$, $k=1,\dots,K$, are a  critical point   of the log-likelihood functional \eqref{eq:intro_log}. Note that in general the solution of \eqref{eq:intro_MFG} is not unique. \\ 
In soft-clustering analysis, the responsibilities can be  used to assign a point to the cluster  with the highest $\g_{k,\e}$, i.e. the   set $\Omega$   is divided into the disjoint  subsets
\begin{equation*}
	S^k_{u,\e}=\{x\in\Omega:\,\g_{k,\e}(x)=\max_{j=1,\dots,K}\g_{j,\e}(x)\}.
\end{equation*}
Taking into account \eqref{eq:intro_resp} and the definition of $m_{k,\e}$ in \eqref{eq:intro_sol}, we see that the clusters $S^k_{u,\e}$ can be equivalently defined as
\begin{equation*} 
	S^k_{u,\e}=\{x\in\Omega:\,u_{k,\e}(x)=\min_{j=1,\dots,K} u_{j,\e}(x)\}.
\end{equation*}
It is well known, in cluster analysis, that the $K$-means functional \eqref{eq:gendist_functional} can be seen as the limit of the maximum likelihood functional   \eqref{eq:intro_log}    when the variance parameter of the  Gaussian mixture model is sent to $0$  (see \cite[Chapter 7]{bishop}).
In order to deduce a PDE characterization for centroidal Voronoi tessellations, we follow a similar idea.  Assuming that $\Sigma_k=\sigma I$ and  passing  to the limit in \eqref{eq:intro_MFG} for $\e,\s\to 0^+$ in such a way that $\e/\s^2\to 1$, we observe that the responsibility $\g_{k,\e}$ in \eqref{eq:intro_resp} converges to the characteristic function of the set where $\a_k m_k$ is maximum with respect to $\a_j m_j$, $j=1,\dots,K$ or,  equivalently,  where $u_k$ is minimum with respect to $u_j$.  Hence, we formally obtain that \eqref{eq:intro_MFG} converges to the first order multi-population MFG system 
\begin{equation}\label{eq:intro_MFG_limit}
	\left\{
	\begin{array}{ll}
		\half |D u_k |^2+\l_k= \half |x-\mu_k|^2 ,\quad& x\in\R^d, \\[8pt]
		\diver(m_k D u_k(x) )=0,&x\in\R^d, \\[8pt] 
		\a_k=\int_{\R^d }\1_{S_u^k}(x) \rho(x)dx,\\[8pt] 
		m_k\ge 0,\,\int_{\R^d} m_k(x)dx=1,    u_k( \m_k)=0,
	\end{array}
	\right.
\end{equation}
for $k=1,\dots,K$, with 
\begin{align}
	&S_u^k=  \{x\in\R^d:\,u_k(x)=\min_{j=1,\dots,K} u_{j}(x)\},\label{eq:intro_cluster_limit}\\ 
	&\mu_{k}=\frac{\int_{\R^d}x \1_{S_u^k}(x)\rho(x)dx}{\int_{\R^d}  \1_{S_u^k}(x)\rho(x)dx}.\label{eq:intro_mean_limit}
\end{align} 
The coupling  among the $K$ systems in \eqref{eq:intro_MFG_limit}  is in the definition of the subsets $S_u^k$ and the coefficient  $\a_k$ represents the fraction of the data set contained in the cluster $S_u^k$.\\ 
In order to write a simplified version of \eqref{eq:intro_MFG_limit}, we observe that the ergodic constant $\l_k$ in the Hamilton-Jacobi equation, which can be characterized as the supremum of the real number $\l$ for which the  equation admits a subsolution (see \cite{bcd}), is  always equal to $0$. Moreover, since the solution $u_k$  is defined up to a constant,  we set $u_k(\mu_k)=0$  and we obtain  $u_k(x)=|x-\mu_k|^2/2$.
The solution, in the sense of distribution, of the second PDE in \eqref{eq:intro_MFG_limit} is given by $m_k=\d_{\m_k}(\cdot)$, where $\d_{\m_k}$ denotes the Dirac function in $\mu_k$. It follows that the HJ equations  are independent of $m_k$ and $\a_k$.   Recalling that the unique viscosity solution of the problem 
\[
\left\{ 
\begin{array}{ll}
	|Du|=1\quad  x\in\R^d, \\[6pt]
	u ( \mu)=0,\\[6pt]
\end{array}
\right.
\]
is given by $u(x)=|x-\mu|$, we can write  the equivalent version of \eqref{eq:intro_MFG_limit}
\begin{equation*} 
	\left\{
	\begin{array}{ll}
		|Du_k|=1\quad  x\in\R^d,  \\[6pt]
		u_k( \m_k)=0,\\[6pt]
		S_u^k=  \{x\in\R^d:\,u_k(x)=\min_{j=1,\dots,K} u_{j}(x)\}, \\[6pt]
		\mu_{k}=\frac{\int_{\R^d}x \1_{S_u^k}(x)\rho(x)dx}{\int_{\R^d}  \1_{S_u^k}(x)\rho(x)dx}
	\end{array}
	\right.
\end{equation*}
for $k=1,\dots,K$,  which is  a system of   HJ  equations coupled through the sets $S_u^k$. Taking into account \eqref{eq:intro_mean_limit}, we see that the previous system coincides with \eqref{eq:gendist_MFG} when $d_C$ is given by the Euclidean distance.

\section{Numerical tests} \label{sec:numerics}
In this section, we present some numerical tests obtained via the approximation of the HJ  systems characterizing the tessellation associated to the problem.\\
 We introduce a regular triangulation of $\O$, the support of $\rho$, given by a collection  of  $N$ disjoint triangles $\mathcal T:=\{T_i\}_{i=1,...,N}$. We denote with $\Delta x$ the maximal area of the triangles, i.e.  $\max_{i=1,...,N}|T_i|<\Delta x$, and we assume that $\Omega\subseteq \bigcup_1^NT_i\approx\Omega$. We denotes with $\mathcal G:=\{X_i\}_{i=1,...,N}$ the set  of the centroids 
of the triangles $T_i$ and, for a piecewise  linear function $U:\mathcal G\to\R$, we set $U_i:=U(X_i)$.
\subsection{Tests for the geodesic $K$-means problem} \label{sec:HJnum}
To test our method, we start with the classical $K$-means problem, i.e. the case where the distance $d_C$ coincides with the Euclidean one. In this case, the Hamiltonian in \eqref{eq:gendist_MFG} is given by $H(x,p) =|p|$ and  the centroids   of the Voronoi diagrams $V(y_k)$ are given by \eqref{centroids}. For the approximation of the HJ equation, we consider the semi-Lagrangian monotone scheme  
$$
G_i(U)=  \min_{a\in B(0,1)}\left\{\mathbb{I}\left[U\right] (X_{i}-ha)+h \right\},
$$
where $h$ is a fictitious-time  parameter (generally taken of order $O(\sqrt{\Delta x})$, see \cite{Festa2017127} for details), and $\mathbb{I}$  a standard linear interpolation operator on the simplices of the triangulation.\\
Our algorithm is a two steps iterative procedure, similar to the Lloyd's algorithm. Starting from an arbitrary assignment $\mu^{(0)}=(\mu^{(0),1},\dots,\mu^{(0),K})$ for the centroids, we iterate
\begin{itemize}
	\item[(i)] For $k=1,\dots,K$ and $i^k=\hbox{argmin}_{i=1,...,N}|X_i-\mu^{(n),k}|$, solve the approximate HJ equations
	\begin{equation}\label{eq:numerics_HJ}
		\left\{
		\begin{array}{ll}
			G_i(U^{(n),k})= 1 ,\quad  i=1,...,N, \\[4pt]
			U^{(n),k}_{i^k}=0,  
		\end{array}
		\right.
	\end{equation}
	where that $U^{(n),k}_{i}$ denotes the value at the $n$-th iteration of the approximate solution of the $k$-th equation at point $X_{i}$, 
	and define 
	$$\cS^{(n+1),k}=  \bigcup\left\{ T_i:\, \hbox{ $i$ is s.t. }\,U^{(n),k}_i =\min_{j=1,...,K}U^{(n),j}_i\right\}.$$
	\item[(ii)] Compute the new centroids points  
	\begin{equation*}
		\mu^{(n+1),k}= \frac{\sum_{T_i\in \cS^{(n+1),k}} X_i |T_i| \rho(X_i)}{\sum_{T_i\in \cS^{(n+1),k}}  |T_i| \rho(X_i)}.  
	\end{equation*}
\end{itemize}
We iterate these two steps till meeting a stopping criterion as 
$$\max_k\{|\mu^{(n+1),k}-\mu^{(n),k}|\}<\varepsilon.$$

\begin{figure}[!t]
	\begin{center}
		\begin{tabular}{cc}
			\hspace{-0.5cm}
			\includegraphics[width=.35\textwidth]{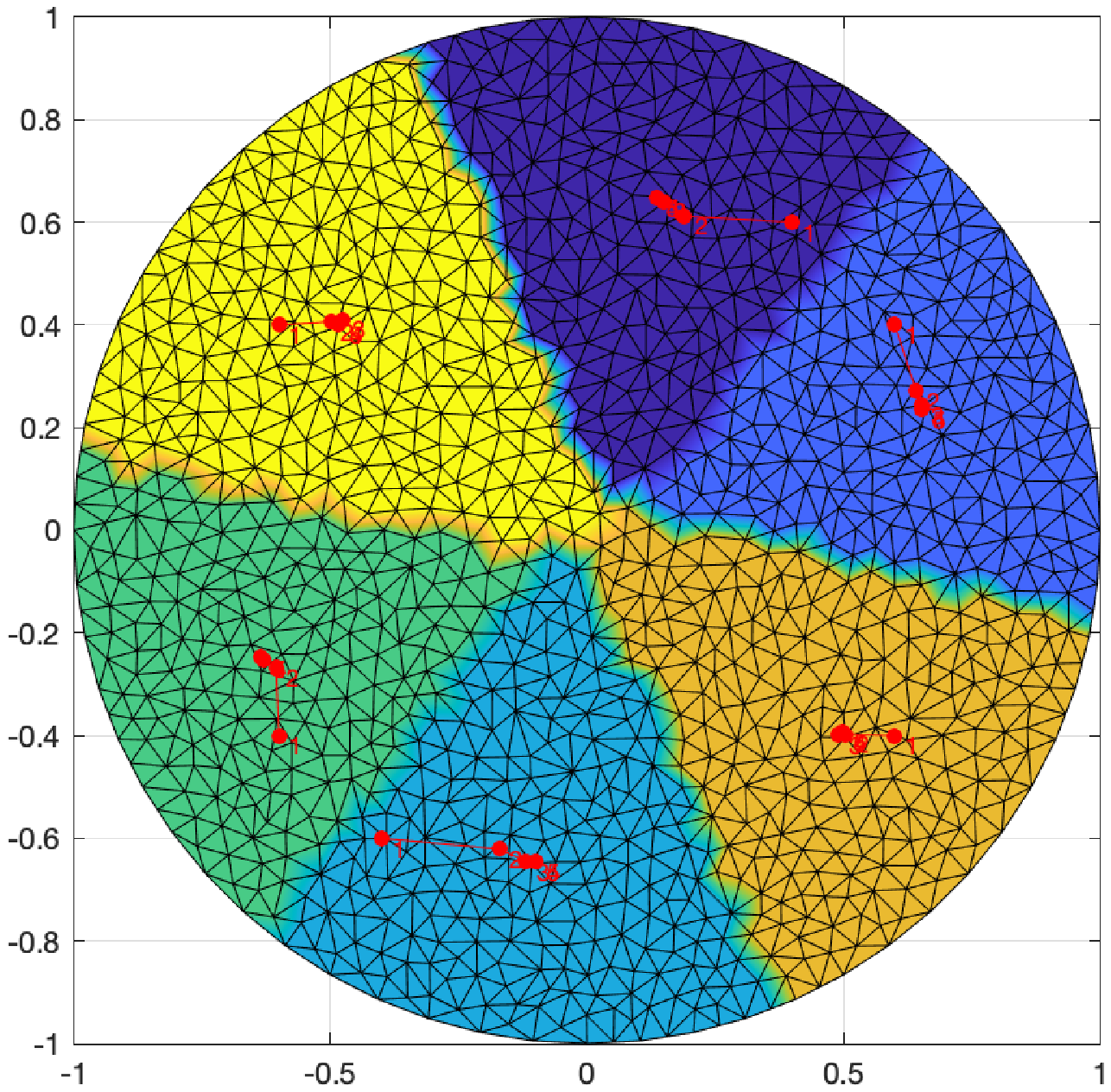}&\hspace{-0.3cm}
			\includegraphics[width=.35\textwidth]{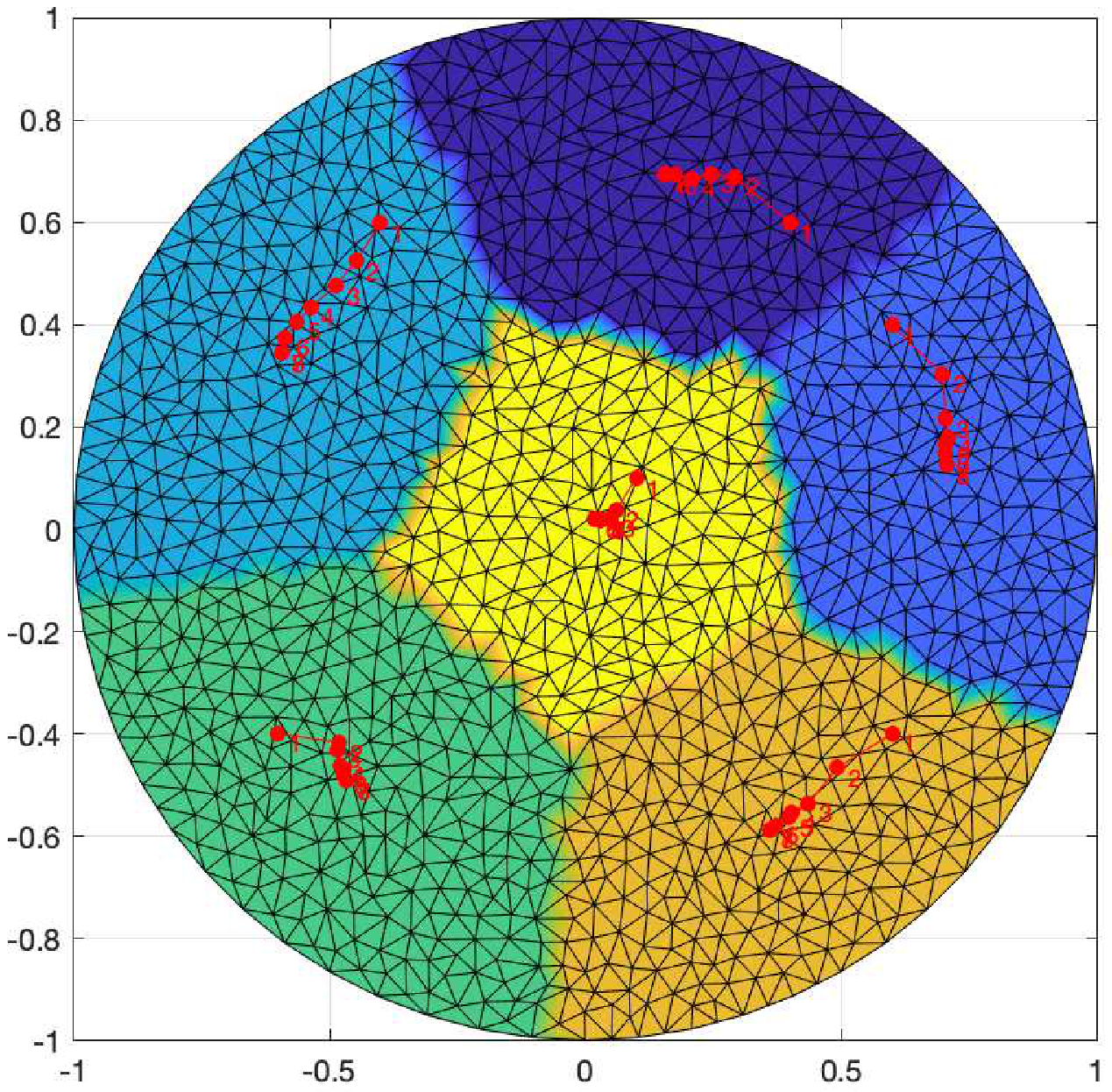}\\
			\includegraphics[width=.35\textwidth]{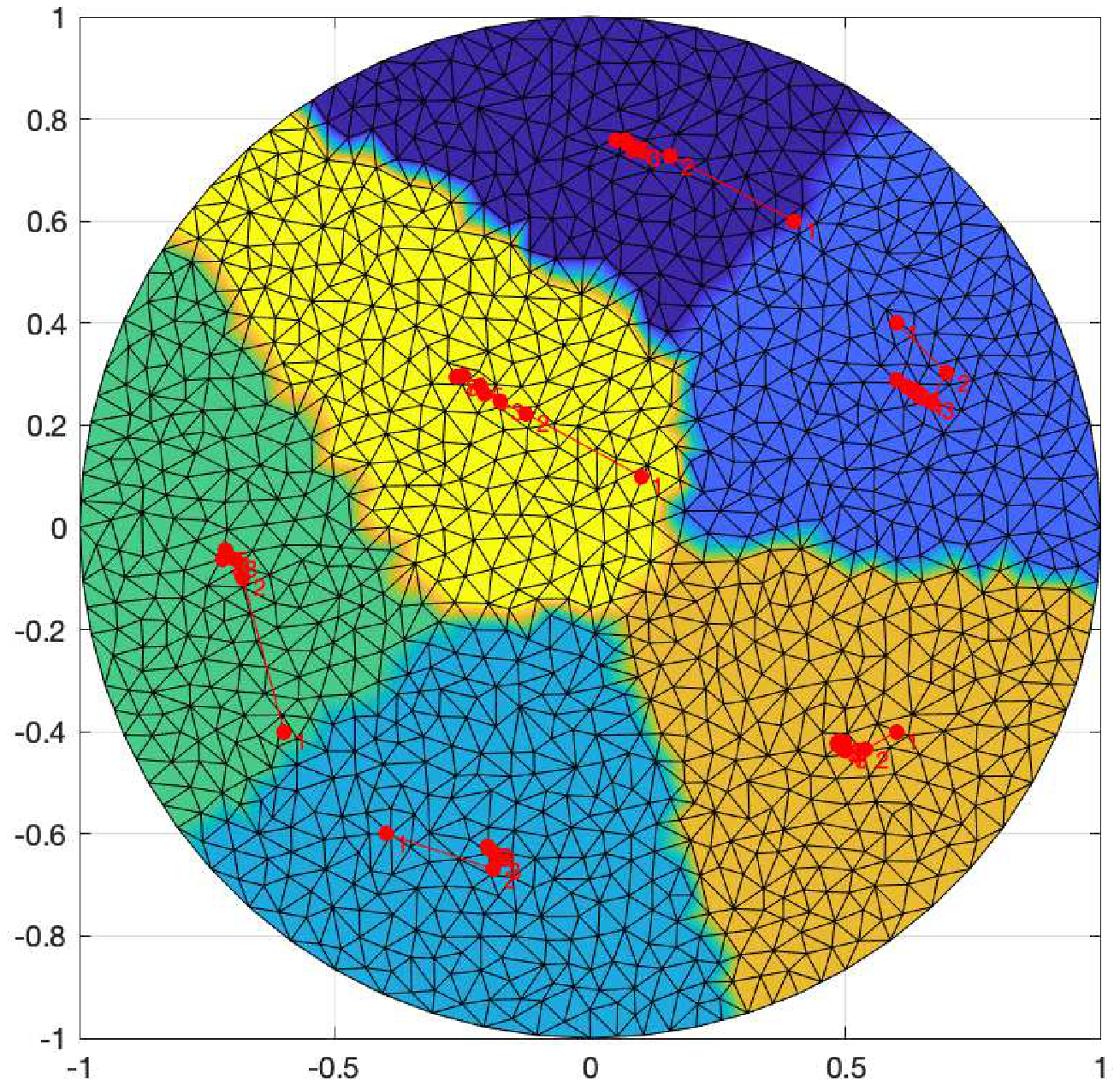}&
			\includegraphics[width=.4\textwidth]{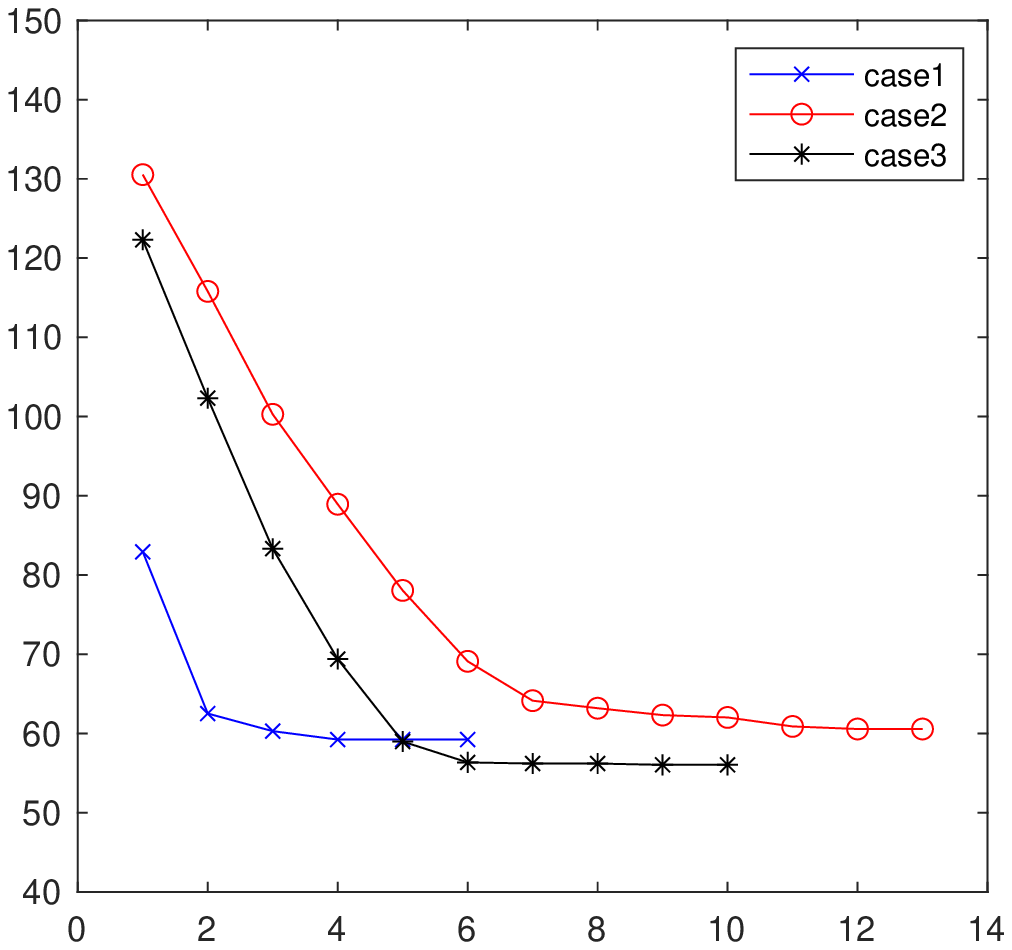}
		\end{tabular}
	\end{center}
	\caption{Three  Voronoi tessellations with $K=6$ computed starting from different initial centroids and  $\Delta x=0.004$, above/left:  {\footnotesize $\mu^{(0)}=( [0.4, 0.6], [0.6, 0.4], [0.6, -0.4], [-0.4, -0.6], [-0.6, -0.4],  [-0.6, 0.4])$};  above/right: {\footnotesize$\mu^{(0)}=( [0.4, 0.6], [0.6, 0.4], [0.6, -0.4],[-0.6, -0.4], [-0.4, 0.6],  [0.1, 0.1])$}; bottom/left: {\footnotesize $\mu^{(0)}=( [0.4, 0.6], [0.6, 0.4], [-0.4, 0.6],[-0.4, -0.6], [-0.6, -0.4],  [0.1, 0.1])$}; bottom/right: evolution of the  $K-$means functional for iteration step of the algorithm.}\label{fig1}
\end{figure}

\paragraph{Test 1.} 
The first test is a simple problem to check the basic features of the technique. We consider a circular domain $\Omega:=B(0,1)$ and we consider a CVT composed of $6$ cells. The density function $\rho$  is chosen uniformly distributed on $\Omega$, i.e.  $ \rho(x)= 1/|\Omega|$, where $|\Omega|=\pi$. We set the approximation parameter  $\Delta x=0.004$  and  the stopping criterion $\varepsilon=\Delta x/10$.
Figure  \ref{fig1} shows tessellations computed by the algorithm   starting from    different sets of initial centroids. 
The evolution of the centroids is marked in red with a sequential number related to the iteration number. We can observe that in all the cases, the centroids move from the initial guess toward an optimal tessellation of the domain, where the optimality is intended referred to the functional \eqref{eq:gendist_functional}. The convergence toward optimality is highlighted in the last picture in Figure~\ref{fig1}, where the value of the K-means functional is evaluated at the end of every iteration for the previous three cases.

\paragraph{Test 2.}

We consider a bounded domain $\Omega$ given by the union of two squares $[0,1]\times[0,1]$, $[-1,0]\times[-1, 0]$ and a section of a circle $B(0,1)\cap[-1,0]\times[0,1]$ and we remove by the domain  the circle $B([-0.4,0.4], 0.2)$,  as displayed in Figure \ref{fig11}. Then, a CVT of $\Omega$ given by three cells, i.e. $K=3$, is computed. 

At first,  the density function $\rho$  is given by  a uniform distribution on $\Omega$, i.e.  $ \rho(x)= 1/|\Omega|$, where $|\Omega|= (2+\pi/4)-\pi(1/5)^2\approx 2.66$. 
\begin{figure}[!t]
	\begin{center}
		\begin{tabular}{cc}
			\hspace{-0.5cm}
			\includegraphics[width=.35\textwidth]{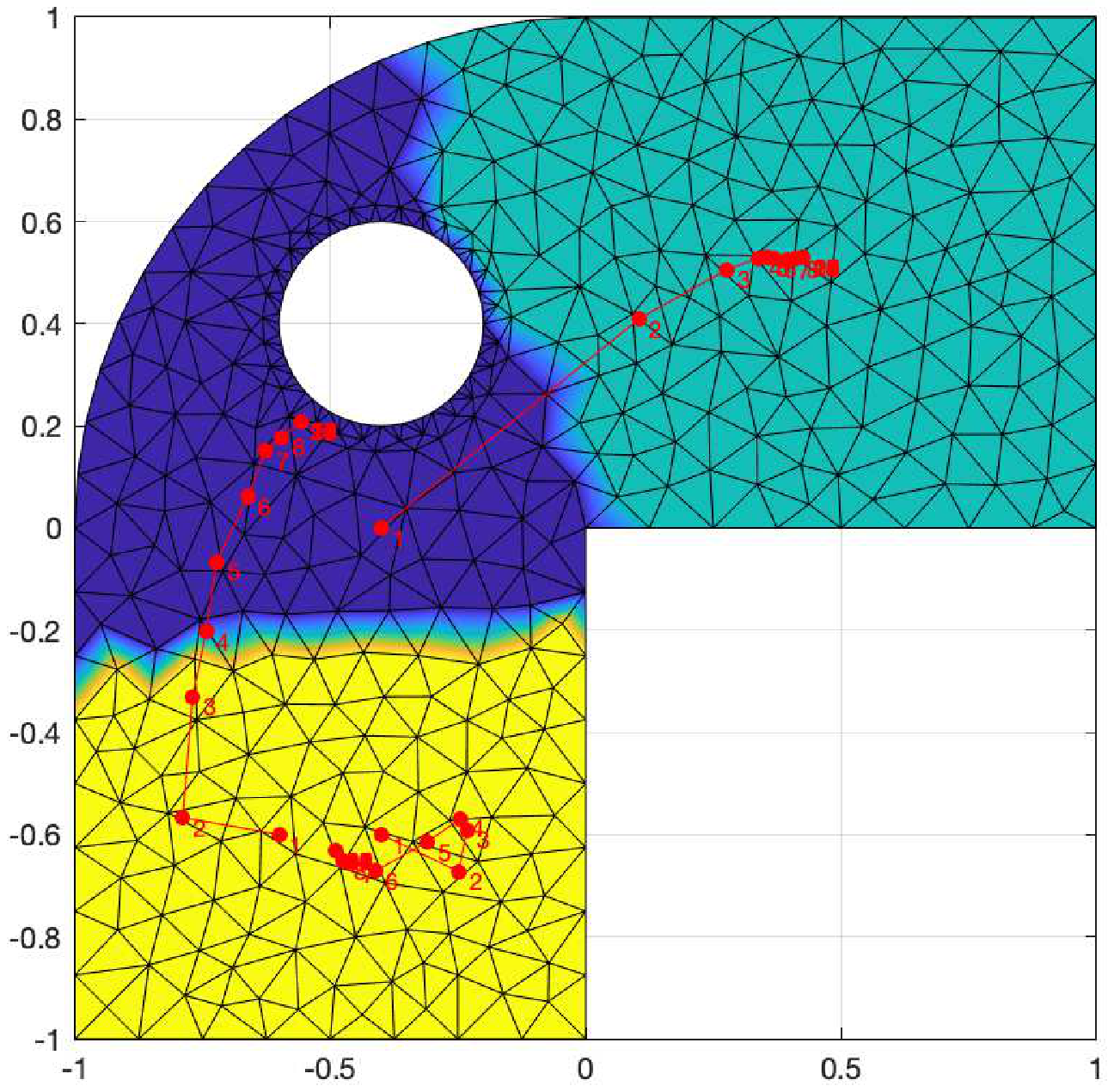}&\hspace{-0.5cm}
			\includegraphics[width=.35\textwidth]{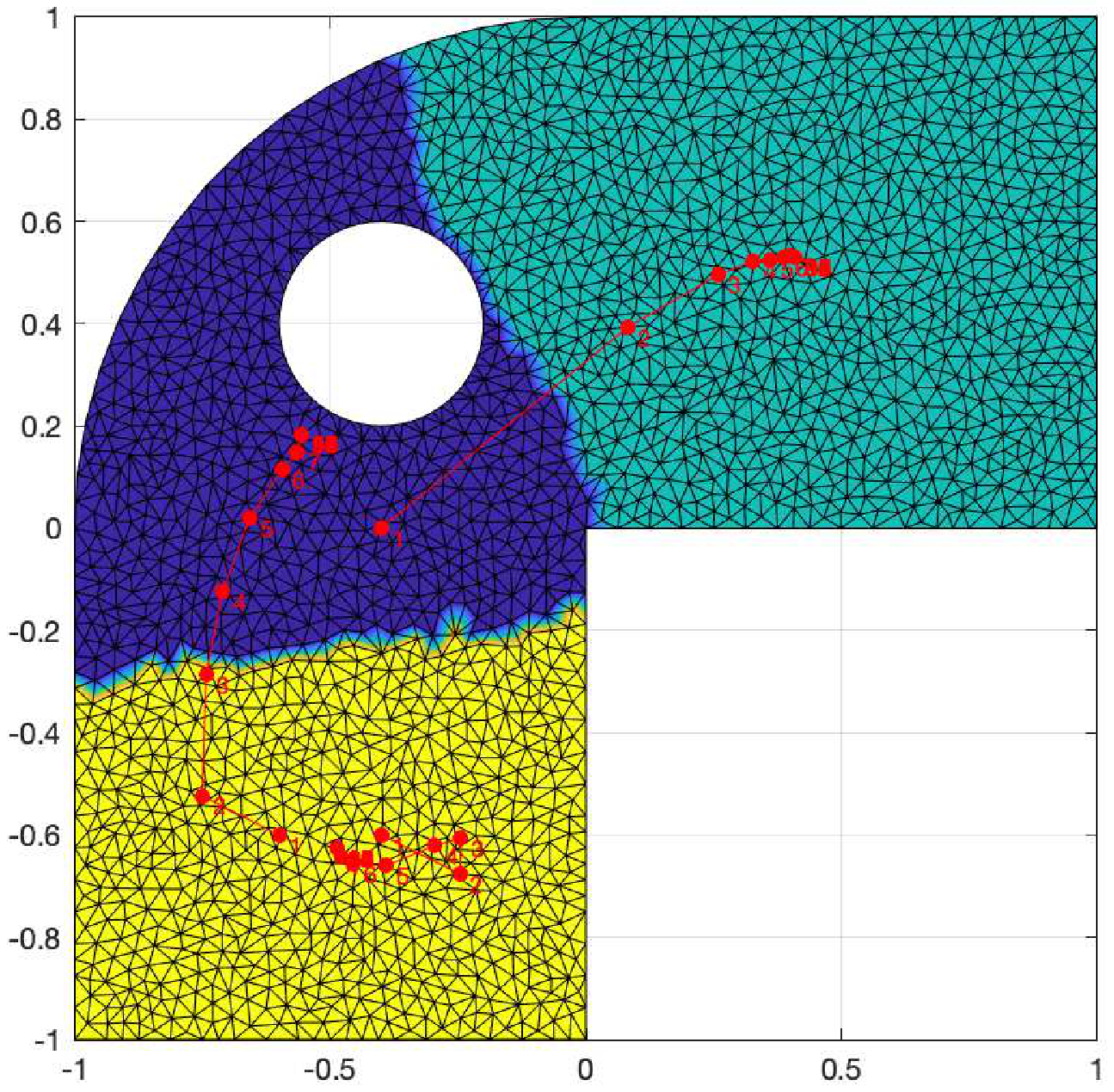}\\
			\includegraphics[width=.35\textwidth]{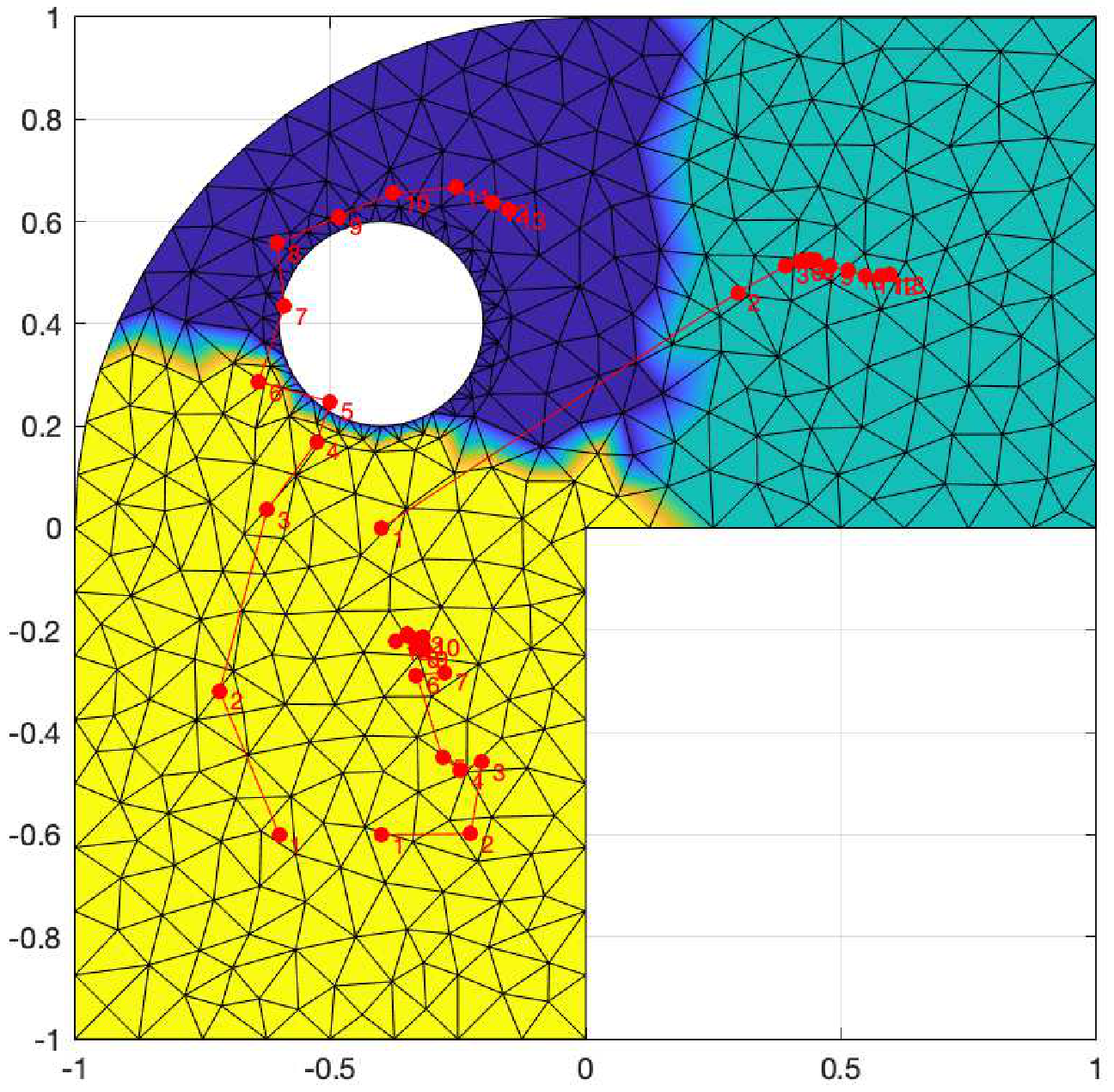}&\hspace{-0.5cm}
			\includegraphics[width=.33\textwidth]{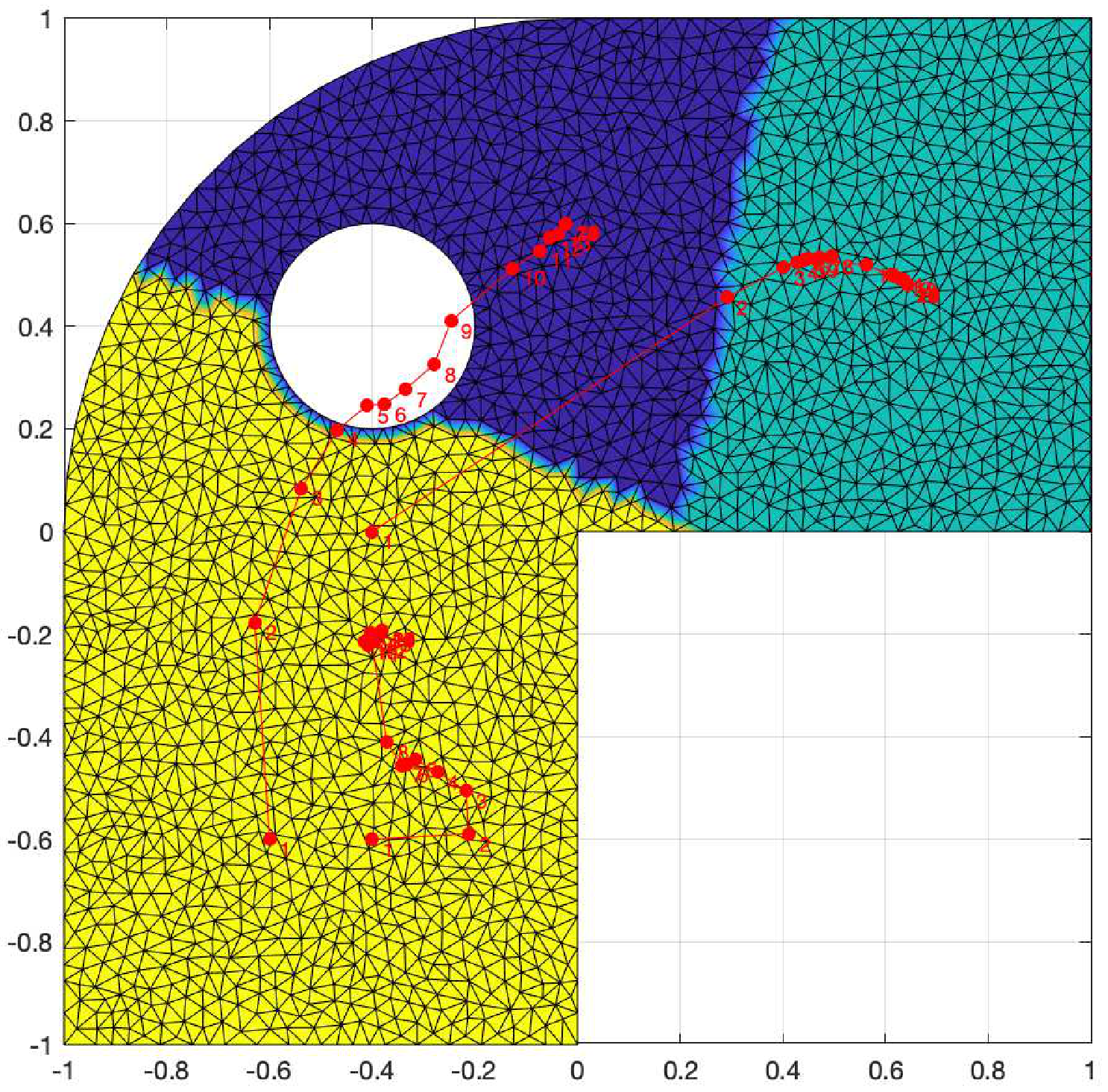}\\
		\end{tabular}
	\end{center}
	\caption{Above: uniformly distributed $\rho$, left: $\Delta x =0.01$,  right: $\Delta x=0.001$. Bottom: $\rho$ is a multivariate normal distribution around $[0.5,0.5]$,  left:  $\Delta x =0.01$;  right: $\Delta x=0.001$.}\label{fig11}
\end{figure}

In Figure  \ref{fig11} we see the evolution  of the centroids $\mu^{(n)}$ starting from the initial position $$\mu^{(0)}=([-0.6,-0.6], [-0.4,-0.6], [-0.4, 0]).$$
The  two images in the top panels of Figure \ref{fig11} are relative to   different discretization parameters $\Delta x:=\max|T_i|$  and  $\varepsilon=\Delta x/10$.  We underline how the number of iterations does not increase much for a smaller stopping parameter, e.g., setting   $\varepsilon$ to $10^{-6}$ we obtained numerical convergence for $n=11$. Moreover, the approximation of the position of the centroids $\mu^{(n)}$, once reached convergence, is sufficiently accurate even in the presence of a discretization parameter $\Delta x$ relatively coarse. This suggests, at least in this example, avoiding excessive refinement of $\Delta x$ to prevent increasing computational cost for the algorithm. 

Even in this easy case, we can observe an additional feature of the method: the approximation of the critical points is monotone with respect to the functional $\cI_{C}$ while a point may have a non-monotone migration toward the correct approximation. This is because the evolution of $\cI_{C}$ in the algorithm is monotone (cf. Fig  \ref{fig1} of the previous test) at any iteration, but not for a single centroid. 

We complete this test with a  case where $\rho$ is not constant. Consider  a multivariate normal distribution around the point $[0.5,0.5]$ and covariance matrix $I$, i.e., 
$$ \rho(x)= \frac{1}{2\pi|\Omega|}e^\frac{{-(x_1-0.5)^2-(x_2-0.5)^2}}{2}.$$
The results are shown in the bottom panels of Figure  \ref{fig11}, with the same choice of the parameters as in the previous test. We observe, as expected, a reduction of the dimension of the sets $\cS^{(n)}_k$ in correspondence to higher values of the density function $\rho$. Even if we need few more steps to reach the numerical convergence, the algorithm shows similar performances and stops for $n=13$.

\begin{remark}
	The previous numerical procedure may be computationally expansive, with the bottleneck given by the resolution of $K$-eikonal equations on the whole domain of interest, see \eqref{eq:numerics_HJ}. In some cases, the first step of Lloyd algorithm may turn to be very expansive, in particular if we use, to solve \eqref{eq:numerics_HJ}, a value iteration method, i.e.  a fixed point iteration on the whole computational domain (see for details \cite{FalcFer}). 
	
	This aspect may be considerably mitigated with the use of a more rational way to process the various parts of the domain, as in the case of Fast Marching methods (see \cite{SethianBook}). In those methods, the nodes of the discrete grid are processed ideally only once, thanks to the information about the characteristics of the problem that may be derived by the same updating procedure. 
	The case of unstructured grids is slightly more complicated than the standard one, and it requires an updating procedure that consider   the geometry of the triangles of the grid. We refer to \cite{SethianVlad} for a precise description of the algorithm in this case.
\end{remark}
We now consider the general case of a geodesic distance $d_C$. As for the Euclidean case, we alternate  the numerical resolution of $K$ HJ equations and updating of the centroids.  To approximate the HJ equation   in \eqref{eq:gendist_MFG}, we consider the semi-Lagrangian scheme     
\begin{equation}\label{semi_lag_scheme}
	G_i(U)=  \min_{\a\in C(X_{i})}\left\{\mathbb{I}\left[U\right] (X_{i}-ha)+h L(X_{i},\alpha)\right\},
\end{equation}
where  $L(x,\a)=\sup_{p\in\R^d}\{p\a-H(x,p)\}$ is the Legendre transform of $H$ (see \cite{FalcFer,Festa2017127}).  
To compute the new centroids,   since  $u_y(x)=d_C(y,x)$,  the   optimization problem in \eqref{eq:gendist_MFG} has its discrete version as
\begin{equation*}
	\sum_{X_j\in \cS^{(n+1),k} }\rho(X_j)d_C(\m_k^{(n+1)},X_j)=
	\min\left\{\sum_{X_j\in \cS^{(n+1),k}}\rho(X_j)d_C(Y,X_j):\,  \text{ $Y\in\cS^{(n+1),k}$}\right\}.
\end{equation*}
and, called $\mathcal{H}(Y)=\sum_{X_j\in\cS^{(n+1),k}}\rho(X_j)d_C(Y,X_j)$, the maximal growth direction is
(see \cite{pc})
\[\delta_k:=D\mathcal{H}(Y)=\half \sum_{X_j\in\cS^{(n+1),k}}\rho(X_j)Dd_C(Y,X_j)n_Y(X_j),\]
where $n_Y(x)$ is the unit vector tangent at $Y$ to the geodesic path joining $X_i$ to $Y$.\\
 Starting from an arbitrary assignment $\mu^{(0)}=(\mu^{(0),1},\dots,\mu^{(0),K})$ for the centroids, we iterate
\begin{itemize}
	\item[(i)] For $k=1,\dots,K$ and $i^k=\hbox{argmin}_{i=1,...,N}|X_i-\mu^{(n),k}|$, solve the problem
	\begin{equation*}
		\left\{
		\begin{array}{ll}
			G_i(U^{(n),k})= 1 ,\quad  i=1,...,N, \\[4pt]
			U^{(n),k}_{i^k}=0,  
		\end{array}
		\right.
	\end{equation*}
	and define 
	$$\cS^{(n+1),k}=  \bigcup\left\{ T_i:\, \hbox{ $i$ is s.t. }\,U^{(n),k}_i =\min_{j=1,...,K}U^{(n),j}_i\right\}.$$
	\item[(ii)]  For $k=1,\dots,K$ compute the new centroids    iterating a gradient descent search. More precisely, fixed a tolerance $\varepsilon_c>0$,  initialize $z_k=\mu^{(n),k}$ and  iterate
	\begin{itemize}
		\item[(a)] Find the value $\alpha_k$ defined as
		\begin{align*}
			\alpha_k&:=\arg\min \mathcal{H}(z_k+\alpha_k\delta_k)\\
			&=\arg\min\sum_{X_j\in \cS^{(n+1),k}}\rho(X_j)d_C(z_k-\alpha_k\delta_k,X_j).
		\end{align*}
		\item[(b)] If $|\alpha_k|<\varepsilon_c$, set $\mu^{(n+1),k}=z_k-\alpha_k\delta_k$, otherwise set $z_k=z_k-\alpha_k\delta_k$ and go back to step (a).
	\end{itemize}
\end{itemize}
We iterate these (i)-(ii) untill meeting a stopping criterion as 
$$\max_k\{|\mu^{(n+1),k}-\mu^{(n),k}|\}<\varepsilon.$$

For the numerical tests, we consider the case of the Minkowski distance on $\R^2$, see Remark \ref{rem:ex_dist}.
We remind that such a distance generalizes the \textit{Manhattan distance} (case $s=1$) and the  \textit{Chebyshev distance} (case $s\rightarrow\infty$, $d_C(x,y)=\max_i(|x_i-y_i|)$). In Figure \ref{figmi} are shown the balls $B(0,1)$ in the Minkowski distance for various values of $s$.

\paragraph{Test 3.}
We first consider the problem on a simple L-shaped bounded domain $\Omega=[0,1]\times[0,1]\cup[-1,0]\times[-1, 1]$ as displayed in Figure \ref{fig4}. In this case, the Chebyshev distance (i.e. $s\rightarrow \infty$) provides an optimal tessellation which is trivially guessed: due to the geometrical characteristics of the domain and the distance (the contour lines of the distance from a point are squares, see Fig. \ref{figmi}) the solution, for $K=3$ and uniform density function ($ \rho(x)= 1/|\Omega|$, where $|\Omega|= 3$),
is  simply composed by the three squares $\{[0,1]^2,[-1,0]\times[-1,0],[-1,0]\times[0,1]\}$, with the centroids $\bar \mu=([0.5,0.5],[-0.5,-0.5],[-0.5,0.5])$.

\begin{figure}[!t]
	\begin{center}
		\begin{tabular}{cc}
			\hspace{-0.5cm}
			\includegraphics[width=.34\textwidth]{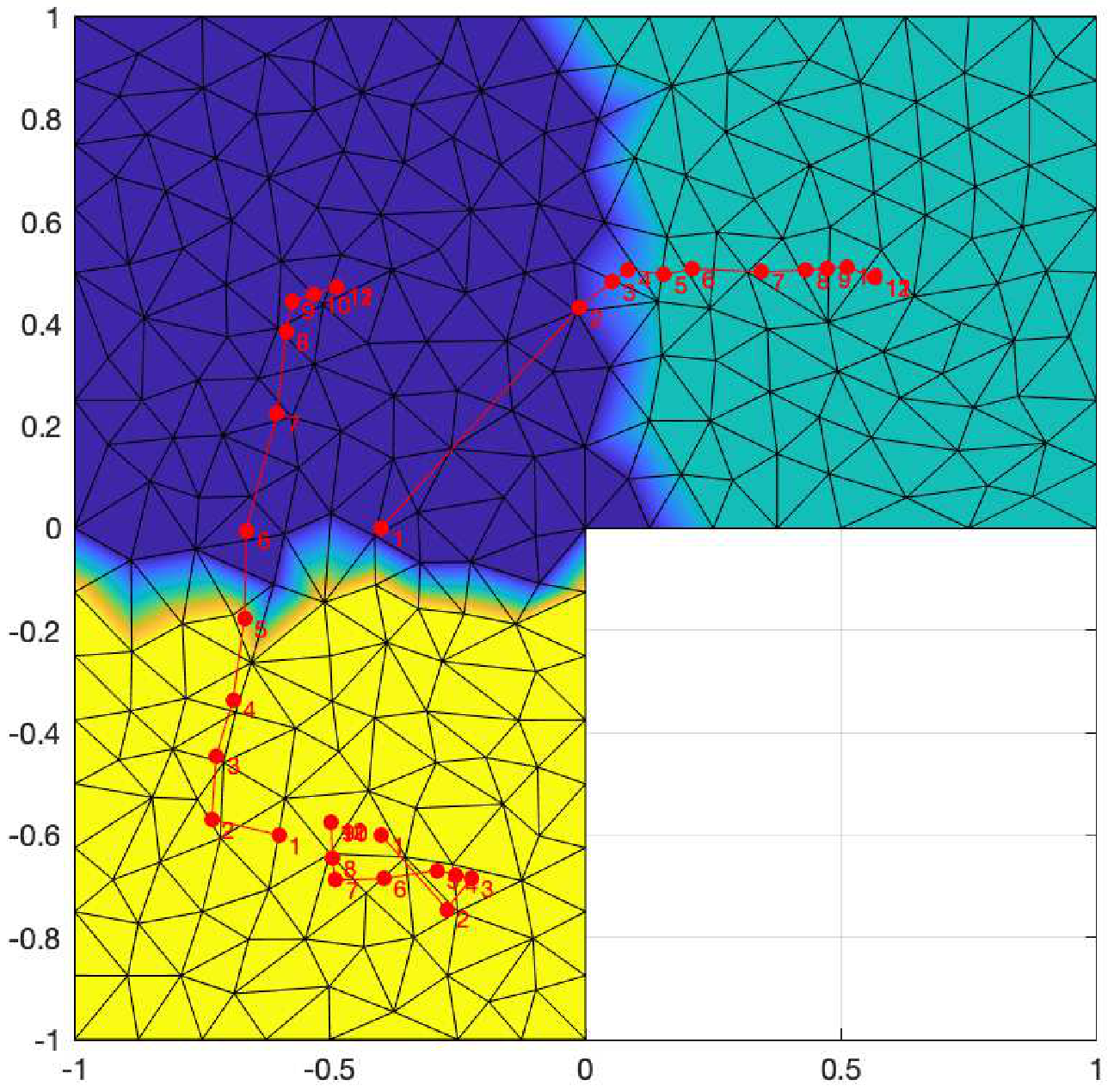}&\hspace{-0.5cm}
			\includegraphics[width=.35\textwidth]{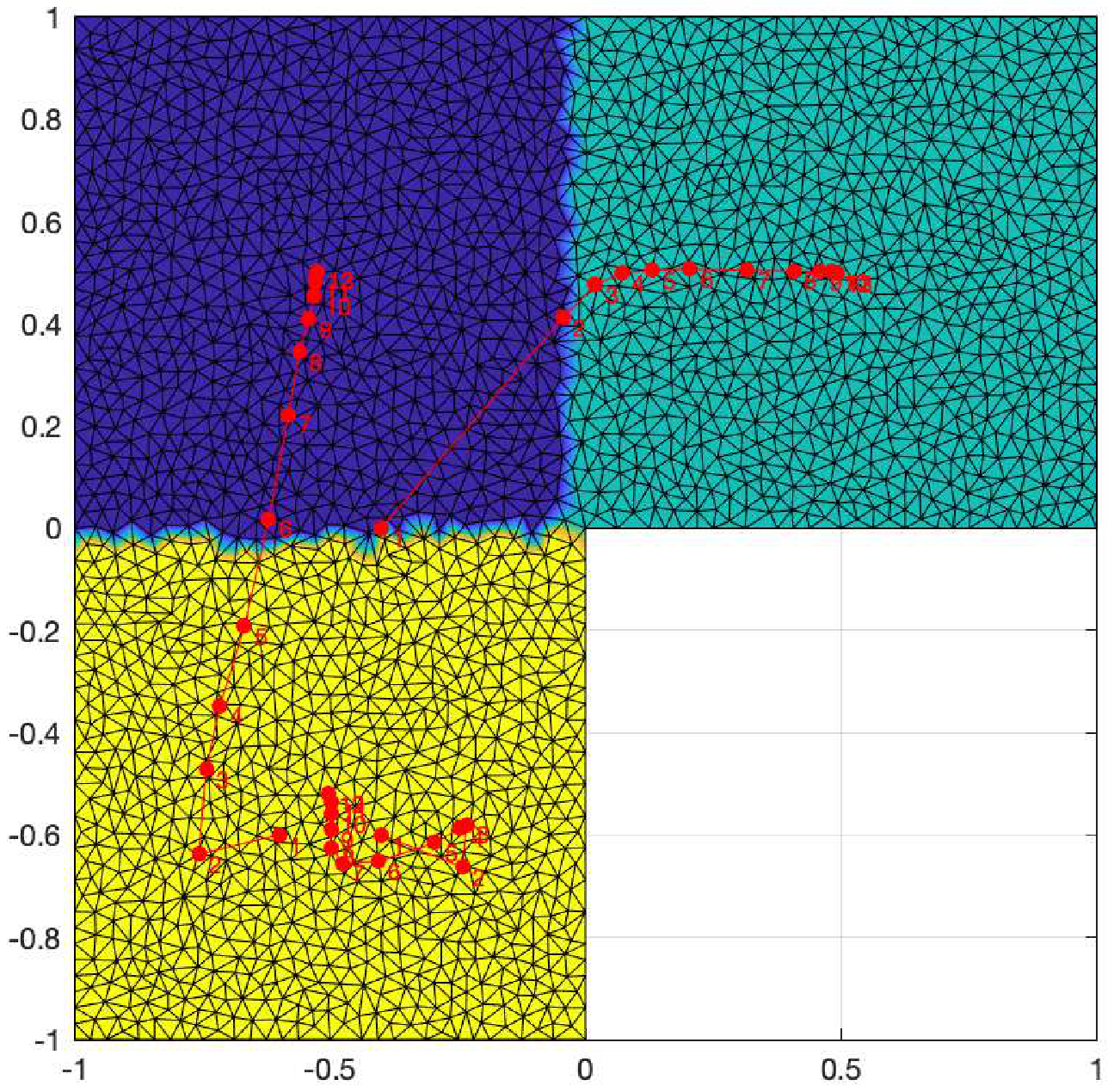}\\
			\includegraphics[width=.35\textwidth]{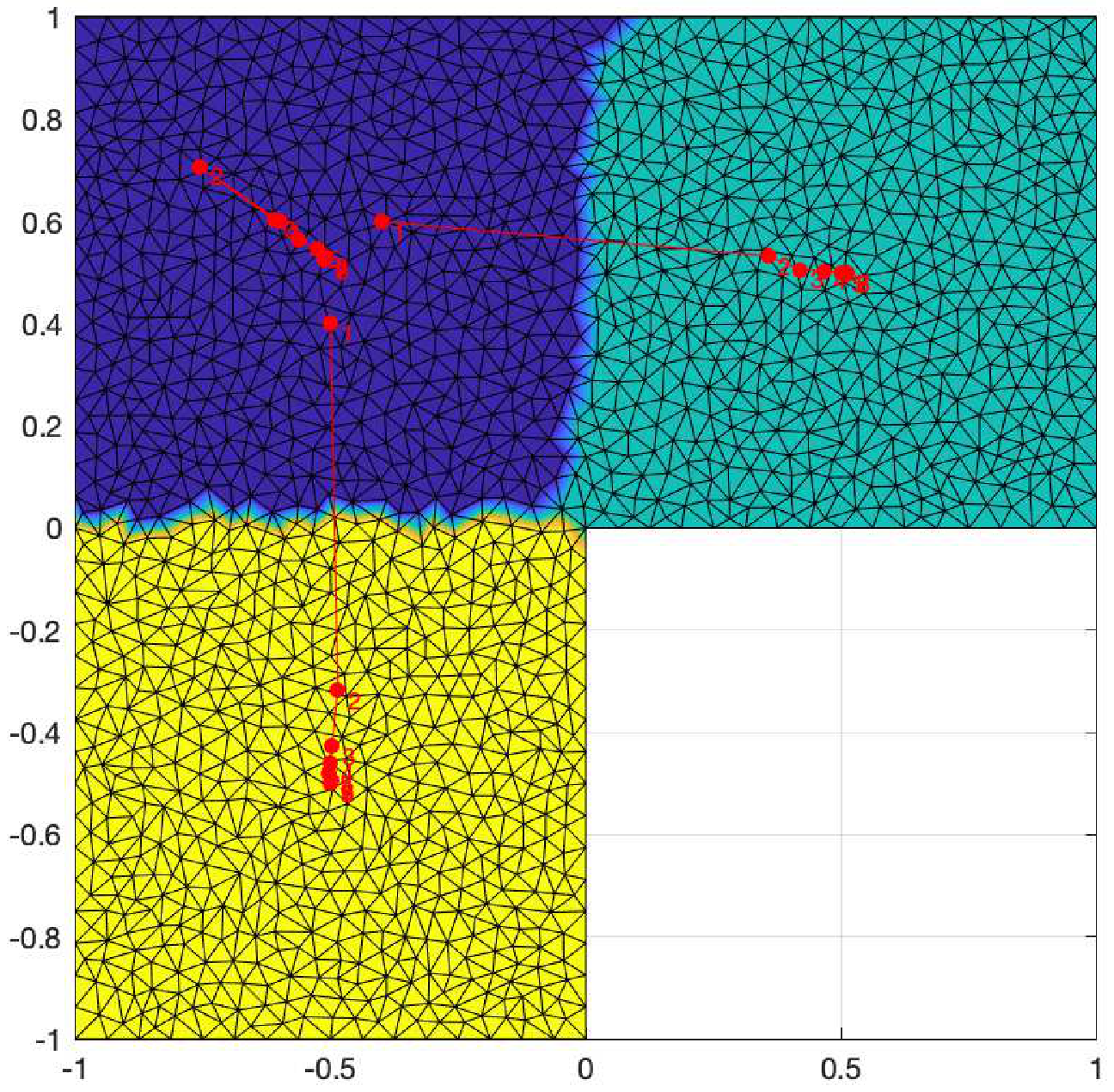}&\hspace{-0.5cm}
			\includegraphics[width=.36\textwidth]{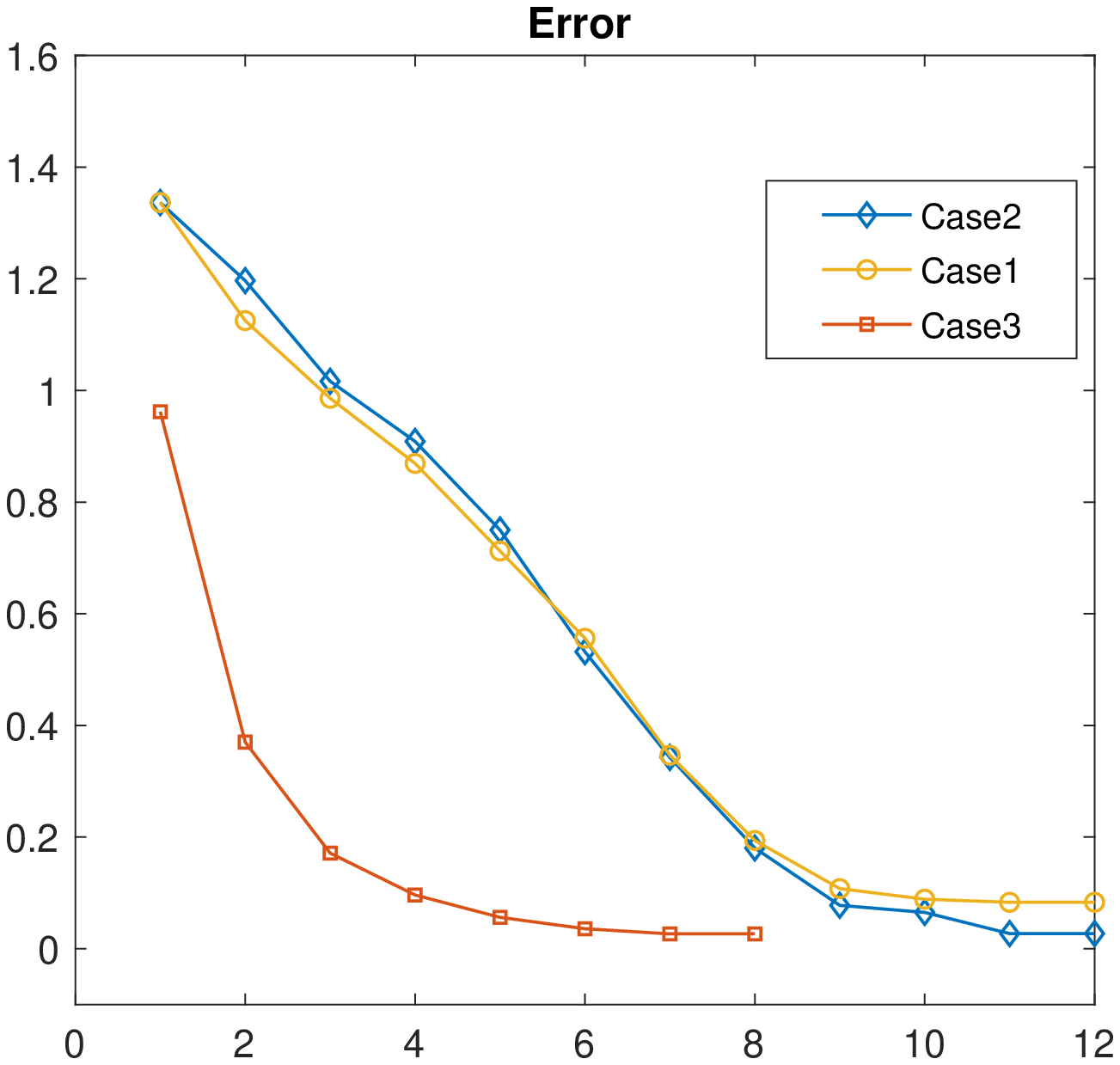}\\
		\end{tabular}
	\end{center}
	\caption{Above: left panel  $\Delta x =0.01$;  right panel: $\Delta x=0.001$ starting from $\mu^{(0)}=([-0.6,-0.6], [-0.4,-0.6], [-0.4, 0])$. Bottom: left panel $\mu^{(0)}=([-0.6,0.6], [-0.4,0.6], [-0.5, 0.4])$, $\Delta x=0.001$, right panel, convergence of the error on the centroids for iterations on the euclidean norm.}\label{fig4}
\end{figure}

In Figure  \ref{fig4}, top panels, we see the evolution  of the centroids $\mu^{(n)}$ starting from the initial position $\mu^{(0)}=([-0.6,-0.6], [-0.4,-0.6], [-0.4, 0])$ for two different values of $\Delta x$.  Also in this case we can observe as the position of the centroids and the rough structure of the tessellation is correctly reconstructed even in the presence of a larger grid. This is also highlighted by the evolution of the euclidean norm of the  error $\mu^{(n)}-\bar \mu$ reported in Fig.  \ref{fig4} (bottom right). The evolution of the error and the total number of iterations necessary to converge to the correct approximation are barely affected by $\Delta x$. On the other hand, the final approximation apparently converges to $\bar \mu$ with order $\Delta x$. 

We perform the same test with a different initial position $\mu^{(0)}$  equal to $([-0.6,0.6], [-0.4,0.6], [-0.5, 0.4])$, see Fig.  \ref{fig4},  bottom left panel. Since in this case the optimal tessellation is unique, the algorithm converges to the same configuration. The number of iterations necessary is clearly affected by the initial guess of $\mu$.\\

\paragraph{Test 4.}
We consider the Minkowski distance with $s=1$. Since contour lines  of the distance  assumes a rhombus shape, we expect to be able to see this in the tessellation that we obtain. In addition, we want to show as our technique, with the help of an acceleration method, can successfully address the GCVT problem with a larger $K$. This is not intended to be an accurate performance evaluation (which is not the main goal of this paper), but only a display of the possibilities give by the techniques proposed. 
\begin{figure}[!t]
	\begin{center}
		\begin{tabular}{cc}
			\hspace{-0.5cm}
			\includegraphics[width=.52\textwidth]{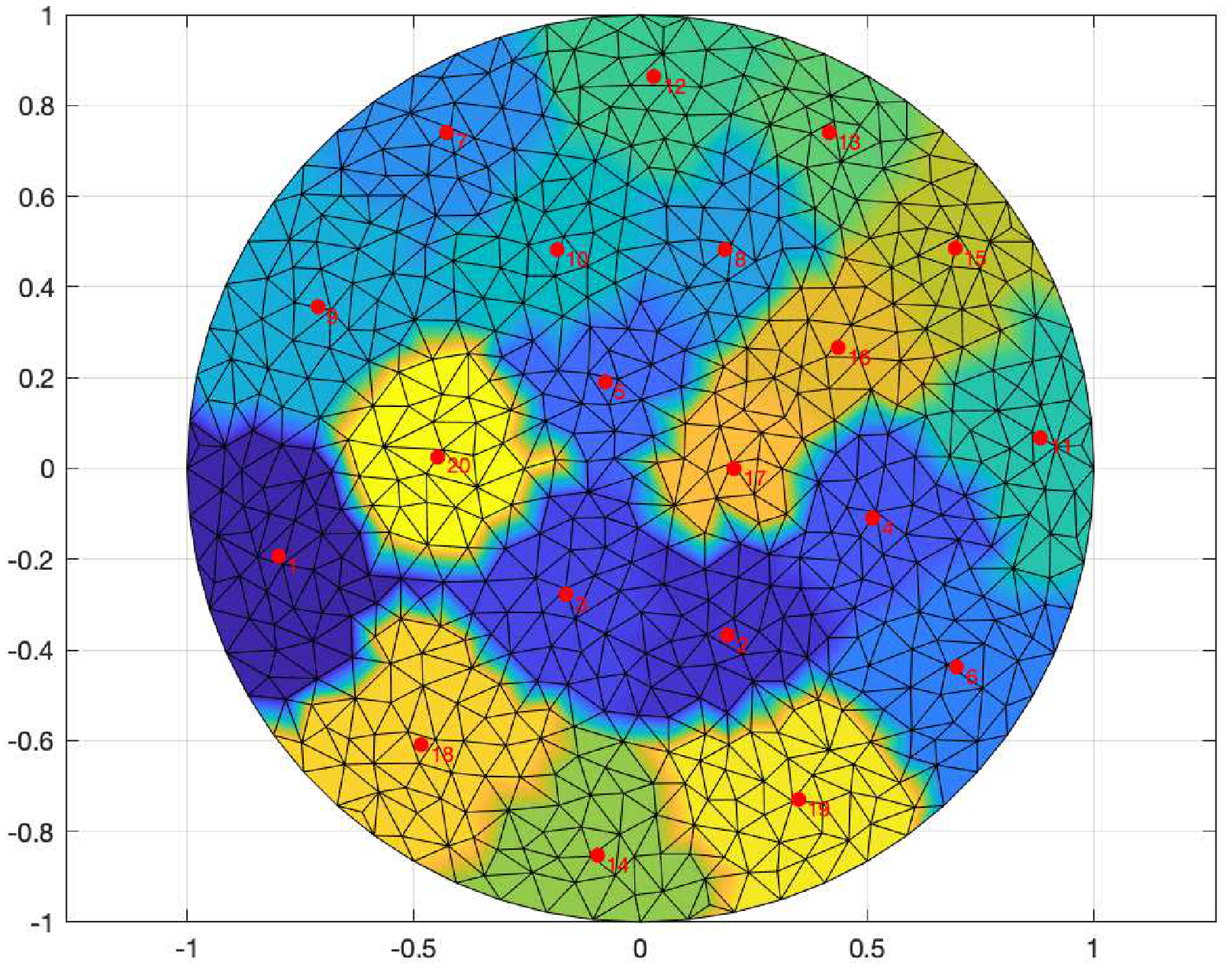}&\hspace{-0.5cm}
			\includegraphics[width=.54\textwidth]{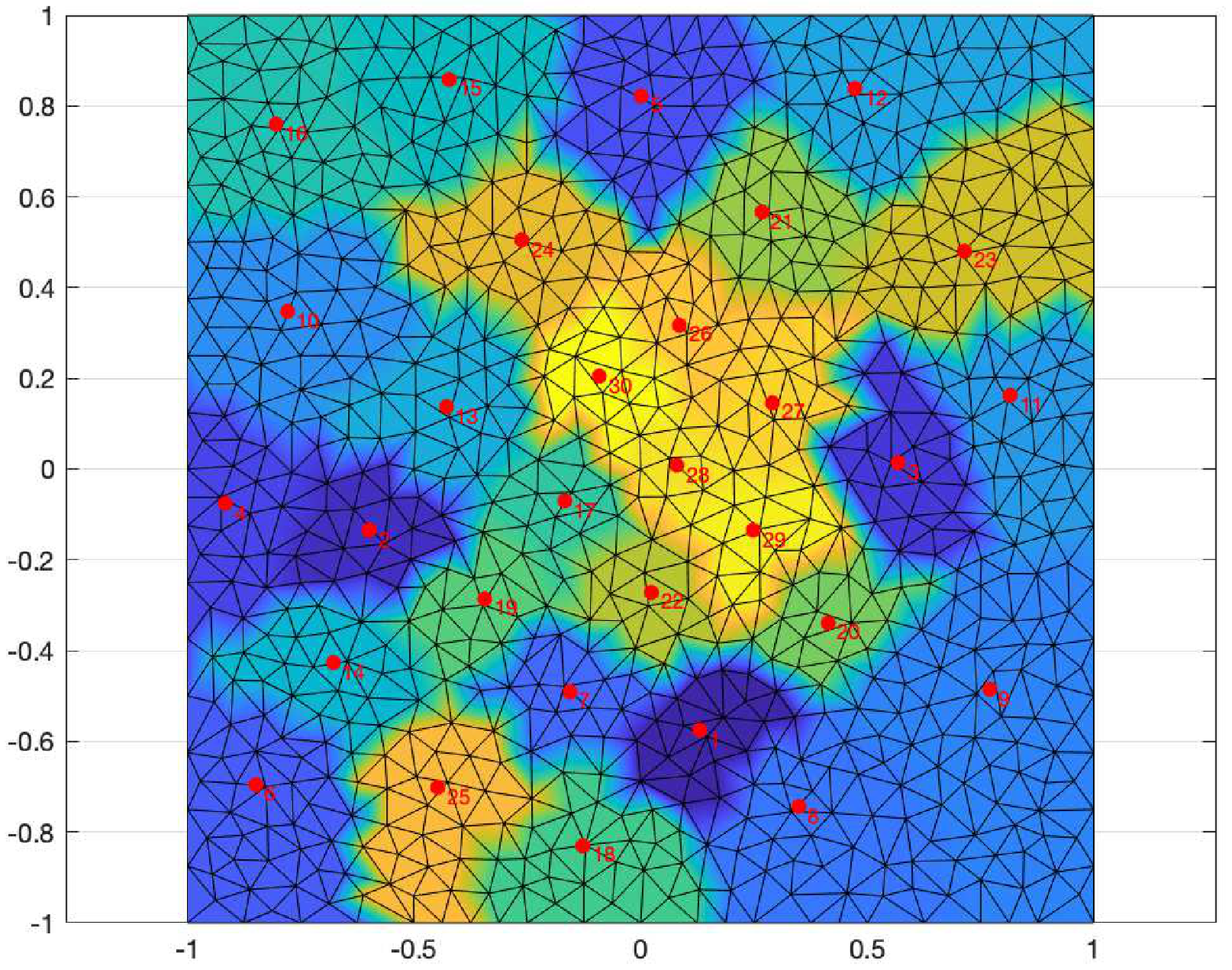}
		\end{tabular}
	\end{center}
	\caption{Two tessellations with larger number of sets. On the left: $\Delta x =0.04$, uniform $\rho$, $K=20$. On the right: $\Delta x =0.04$, multivariate $\rho$ around the origin, $K=30$. In this figure, the numbers are merely to identify the k-centroid of the tessellation.}\label{fig20}
\end{figure}
We consider tesselations of $\Omega=B(0,1)$ with $K=20$ and of $\Omega=[-1,1]^2$ with $K=30$. In the first case (the circle), the function $\rho$  is given by  a uniform distribution, while, in the second case, by multivariate normal distribution around the point $[0,0]$ and covariance matrix $I$, i.e., 
$$ \rho(x)=\rho([x_1,x_2])= \frac{1}{2\pi|\Omega|}e^\frac{{-(x_1)^2-(x_2)^2}}{2},$$
where $|\Omega|=4$. The resulting  tessellations are shown in Figure \ref{fig20}. 

We see that our technique can address without too much troubles a problem with an higher $K$: indeed, this parameter enters in the first step of the Lloyd algorithm linearly. Since we did not observe a substantial change of the number of iterations of the algorithm for a larger $K$, the technique remains computationally feasible, even performed on a standard laptop computer.

\subsection{Tests for  geodesic centroidal power diagrams}
The procedure to obtain an approximation of centroidal power diagrams contains all the tools already described in the previous sections and  it includes a three steps procedure: resolution of $K$ HJ equations,   update of the centroids points and   optimization step for the weights. Starting from an arbitrary assignment $(\mu^{(0)},w^{(0)})=(\mu^{(0),1},\dots,\mu^{(0),K},w^{(0),1},\dots,w^{(0),K})$ for the centroids and   the weights, we iterate
\begin{itemize}
\item[(i)] 
For $k=1,\dots,K$ and $i^k=\hbox{argmin}_{i=1,...,N}|X_i-\mu^{(n),k}|$, solve the problem
\begin{equation*}
	\left\{
	\begin{array}{ll}
		G_i(U^{(n),k})= 1 ,\quad  i=1,...,N, \\[4pt]
		U^{(n),k}_{i^k}=-w_k^{(n-1)},  
	\end{array}
	\right.
\end{equation*}
where $G$ is as in \eqref{semi_lag_scheme}, and define 
$$\cS^{(n+1),k}=  \bigcup\left\{ T_i:\, \hbox{ $i$ is s.t. }\,U^{(n),k}_i =\min_{j=1,...,K}U^{(n),j}_i\right\}.$$

\item[(ii)] Compute the new centroids points  
\begin{multline*}
	\sum_{X_j\in \cS^{(n+1),k} }\rho(X_j)d_C(\m_k^{(n+1)},X_j)
	\\
	=
	\min\left\{\sum_{X_j\in \cS^{(n+1),k}}\rho(X_j)d_C(Y,X_j):\,  \text{ $Y\in\cS^{(n+1),k}$}\right\}.
\end{multline*}
\item[(iii)]  Compute the new weights $w^{(n+1),k}$ as local maximum of the Lagrangian function
\begin{align*}
L(Y,w)=\sum_{i=1}^k \sum_{X_j\in \cS^{(n+1),k}}\rho(X_j)d_C(Y,X_j)-\sum_{i=1}^k w_k( \pi(\cS^{(n+1),k})-c_k) 
\end{align*}

\end{itemize}
We iterate these three steps till meeting a stopping criterion as 
$$\max\{|\mu^{(n+1),k}-\mu^{(n),k}|, |\o^{(n+1),k}-\o^{(n),k}|\}<\varepsilon.$$

\begin{figure}[!t]
\begin{center}
\begin{tabular}{cc}
\hspace{-0.5cm}
\includegraphics[width=.33\textwidth]{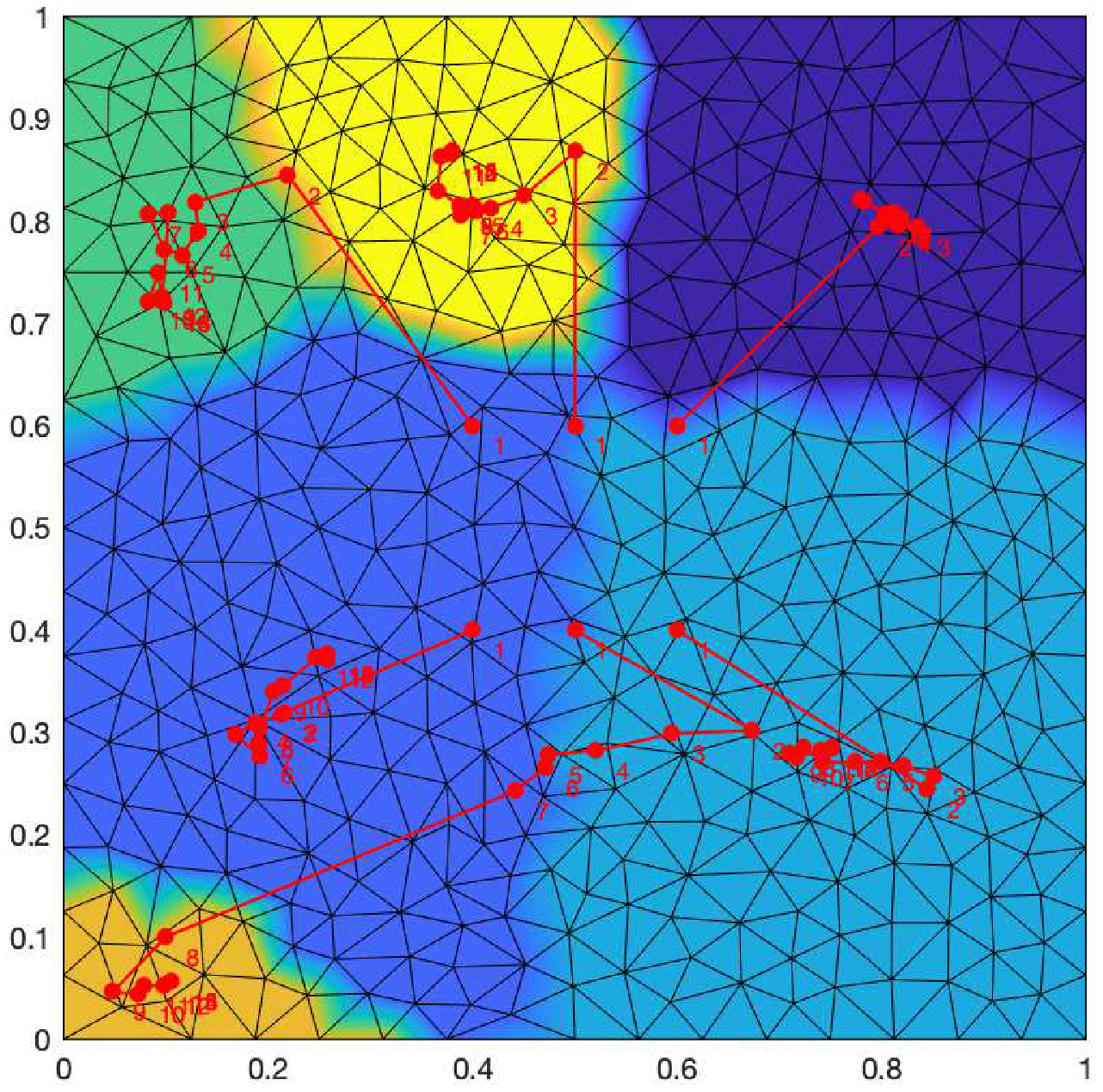} 
\includegraphics[width=.33\textwidth]{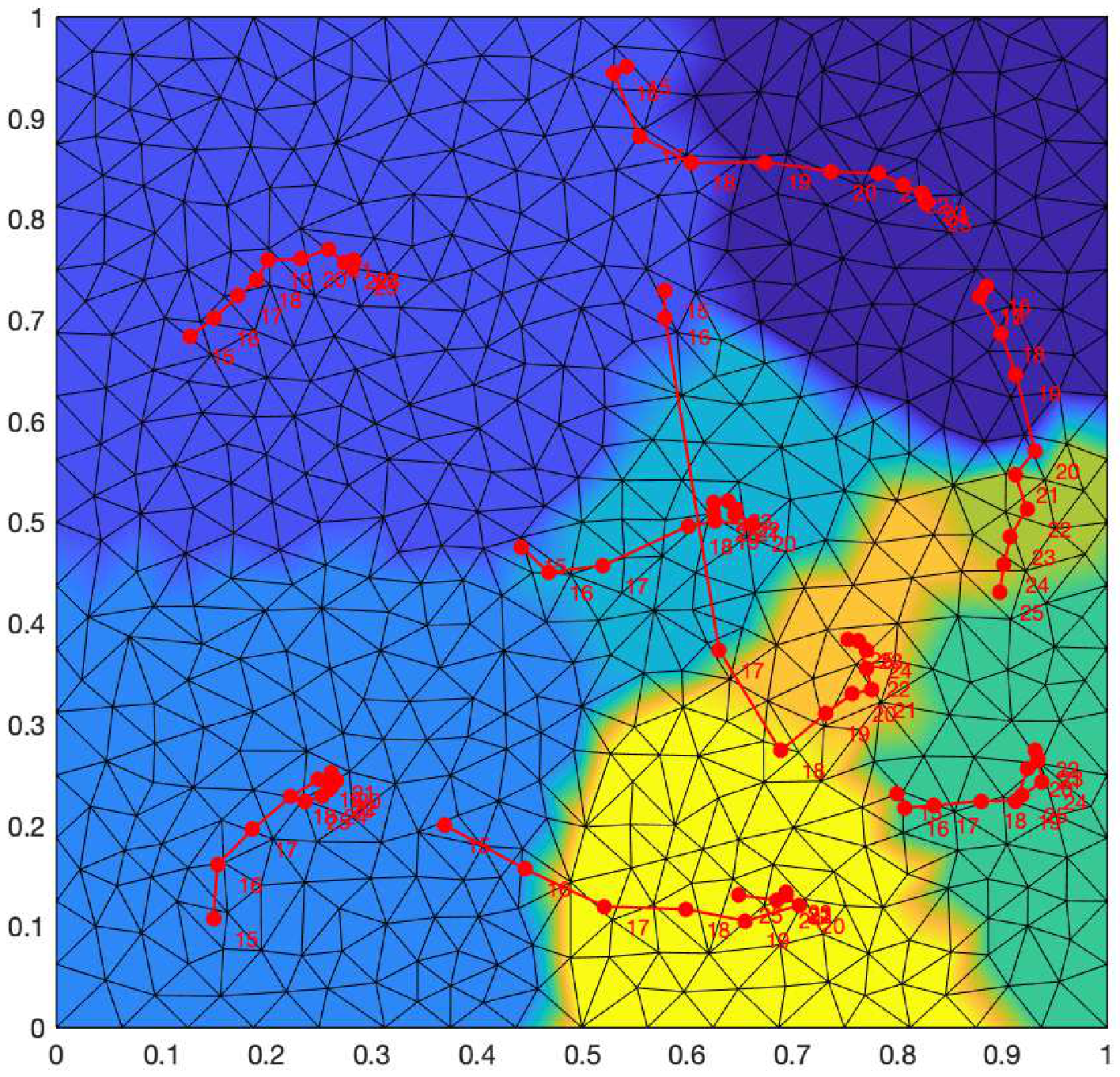}
\end{tabular}
\end{center}
\caption{left panel: $K=6$, $\Delta x =0.015$, $\mu^{(0)}=\{0.4, 0.5, 0.6\}\times\{0.4, 0.6\}$;  right panel: $\Delta x=0.001$ $\mu^{(0)}=\{0.4, 0.45, 0.55, 0.6\}\times\{0.4, 0.6\}$, $c=(0.3, 0.24, 0.15, 0.1, 0.08, 0.06, 0.05, 0.02)$ {\small (in this image some points of the evolution of the centroids are omitted for a better clarity)}.}\label{fig5}
\end{figure}

\paragraph{Test 5.}
We test the centroidal power diagram procedure, Figure  \ref{fig5}, in a simple case given by the unitary square $\Omega=[0,1]^2$
 for $K=6,8$ and capacity constraint given respectively   by
 \begin{align*}
 	&c=(0.3,0.25, 0.18, 0.12, 0.1, 0.05),\\
 	&c=(0.3, 0.24, 0.15, 0.1, 0.08, 0.06, 0.05, 0.02).
 \end{align*}
Clearly we have $\bigcup_k \cS^{(n),k}=\Omega$, for any $n$ and therefore we $\sum_k c_k=|\Omega|=1$. \\

The same technique is used to generate some power diagrams of more complex domains: in Figure  \ref{fig6}  we show the optimal tessellation of a text and a rabbit-shaped domain. In the first case, the algorithm parameters are set to $K=8$, $c=(0.33, 0.22, 0.1, 0.1, 0.1, 0.05, 0.05, 0.05)$, $\Delta x=0.002$. In the second one, $K=6$,  $c=(0.3, 0.15, 0.15, 0.15, 0.15, 0.10)$,  $\Delta x=0.002$.
\begin{figure}[!t]
\begin{center}
\begin{tabular}{cc}
\hspace{-0.5cm}
\includegraphics[width=0.8\textwidth]{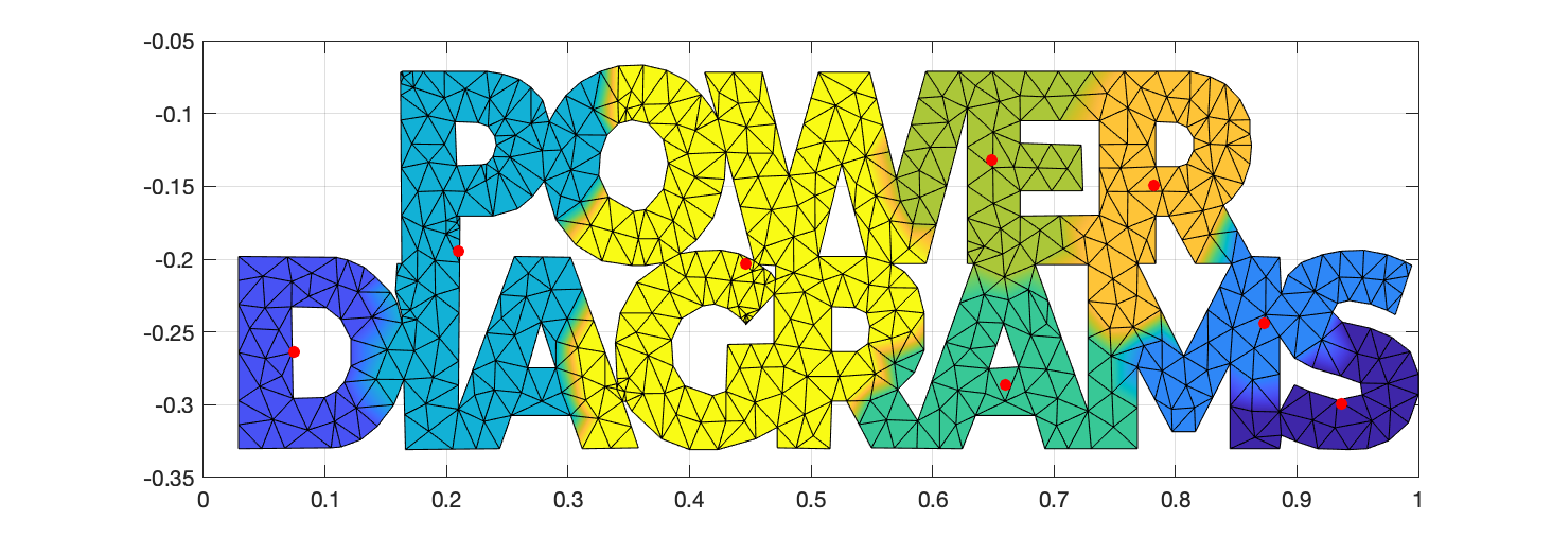}\\
\includegraphics[width=.55\textwidth]{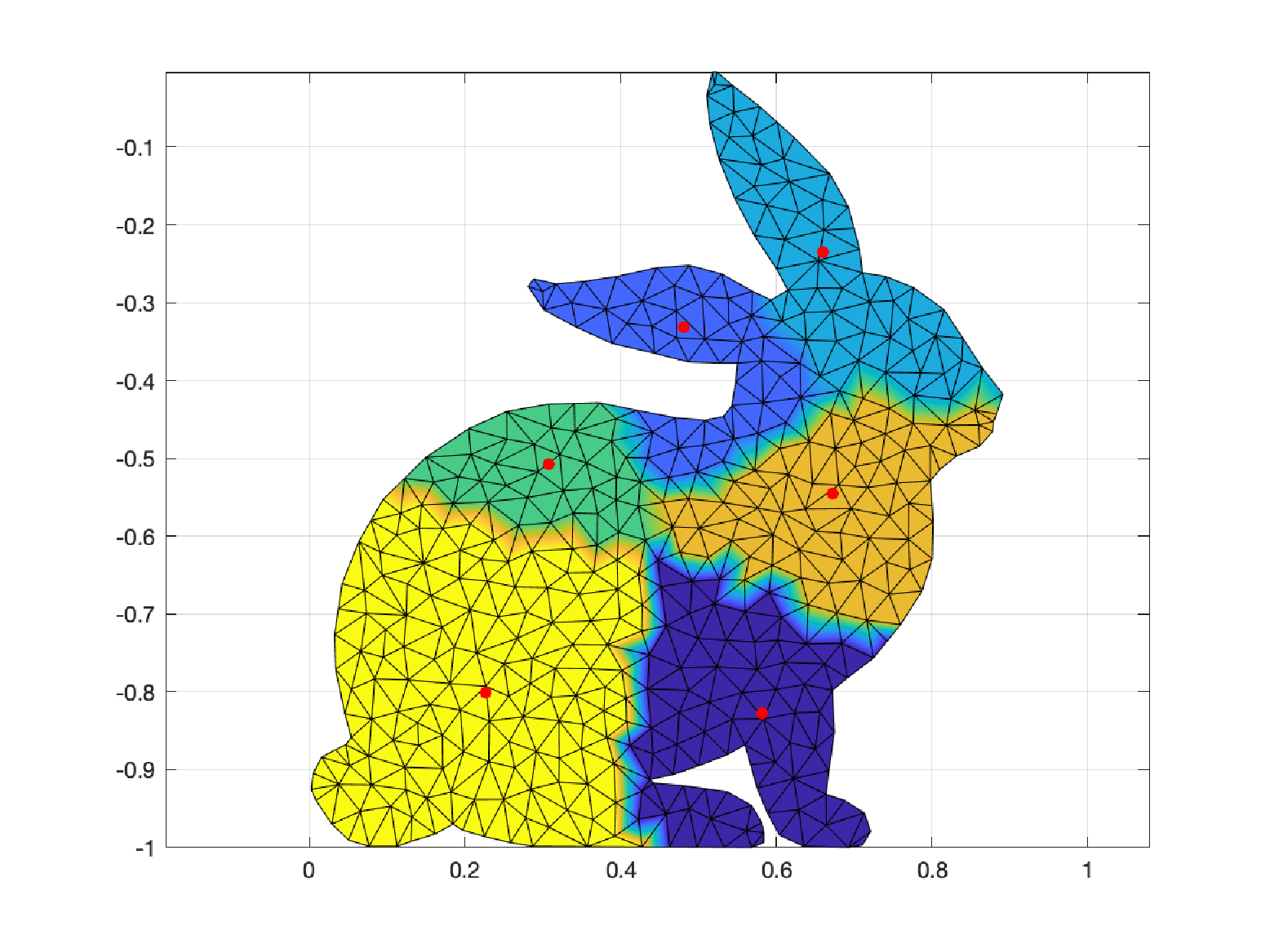}
\end{tabular}
\end{center}
\caption{Optimal power diagrams of a text and a rabbit shaped domain. The parameters are set $K=8$, $c=(0.33, 0.22, 0.1, 0.1, 0.1, 0.05, 0.05, 0.05)$, $\Delta x=0.002$ (left) $K=6$,  $c=(0.25, 0.15, 0.15, 0.15, 0.15, 0.10)$,  $\Delta x=0.002$ (right).}\label{fig6}
\end{figure}



\begin{thebibliography}{99}
\bibitem{aha}
F. Aurenhammer, F. Hoffmann,  B.  Aronov, {\it Minkowski-type theorems and least-squares
clustering}, Algorithmica  20 (1998), 61-76.

\bibitem{akl}
F. Aurenhammer, R.  Klein, D.T. Lee, {\it Voronoi Diagrams and Delaunay Triangulations.} 
World Scientific Publishing Company, Singapore, 2013.

	
\bibitem{accd}
L. Aquilanti, S. Cacace, F.  Camilli, R. De Maio, {\it  A mean field games approach to cluster analysis}.  Appl. Math. Optim. 84 (2021), no. 1, 299-323.

\bibitem{bcd}
M. Bardi, I. Capuzzo Dolcetta, {\it Optimal Control and Viscosity Solutions of Hamilton-Jacobi-Bellman equations}. Birkhäuser, Boston, 1997.

\bibitem{bishop}
C. M. Bishop, \textit{Pattern recognition and Machine Learning}. Information Science and Statistics, Springer, New York, 2006.


\bibitem{bourne}
D. P. Bourne, S.M. Roper, {\it Centroidal power diagrams, Lloyd's algorithm, and applications to optimal location problems. } SIAM J. Numer. Anal. 53 (2015), no. 6, 2545-2569. 



\bibitem{cdi}
I. Capuzzo Dolcetta, {\it A generalized Hopf-Lax formula: analytical and approximations aspects.  Geometric control and nonsmooth analysis}, 136-150, Ser. Adv. Math. Appl. Sci., 76, World Sci. Publ., Hackensack, NJ, 2008.

\bibitem{cd}
R. Carmona, F. Delarue, {\it Probabilistic theory of mean field games with applications. I. Mean field FBSDEs, control, and games.} Probability Theory and Stochastic Modelling, 83. Springer, Cham, 2018.


\bibitem{degoes}
 F. De Goes, K. Breeden, V. Ostromoukhov, M. Desbrun,
 {\it Blue noise through optimal transport}. ACM Transactions on Graphics 31 (2012), art. no. 171.


\bibitem{dfg_rev}
Q. Du, V. Faber, M. Gunzburger, {\it Centroidal Voronoi tessellations: applications and algorithms}. SIAM Rev. 41 (1999), no. 4, 637-676.

\bibitem{dfg_sinum}
Q. Du, M.  Emelianenko, L. Ju,
{\it Convergence of the Lloyd algorithm for computing centroidal Voronoi tessellations.  }
SIAM J. Numer. Anal. 44 (2006), no. 1, 102-119.

\bibitem{FalcFer}
M. Falcone, R. Ferretti, {\it Semi-Lagrangian approximation schemes for linear and Hamilton-Jacobi equations}. Society for Industrial and Applied Mathematics (SIAM), Philadelphia, PA, 2014.  

\bibitem{Festa2017127}
A. Festa, R. Guglielmi, R. Hermosilla, A. Picarelli, S. Sahu, A. Sassi, F.J. Silva,  Hamilton-Jacobi-Bellman equations.
\newblock in {\it Optimal control: novel directions and applications}, 127-261, Lecture Notes in Math., 2180, Springer, Cham, 2017.

\bibitem{Festa16}
A. Festa, {\it Reconstruction of independent sub-domains for a class of Hamilton-Jacobi equations and application to parallel computing.} ESAIM: Mathematical Modelling and Numerical Analysis 50(4), (2016) 1223 -- 12401.




\bibitem{gs}
D. Gomes, J. Saude,  Mean field games - A brief survey. Dyn. Games Appl. 4 (2014), no. 2,
110-154.


\bibitem{Kalise18}
D. Kalise, K. Kunisch,
{\it Polynomial approximation of high-dimensional Hamilton–Jacobi–Bellman equations and applications to feedback control of semilinear parabolic PDES} SIAM Journal on Scientific Computing, 40 (2), (2018)  A629-A652.

\bibitem{ll}
J.M. Lasry, P.L. Lions, {\it Mean Field Games}, Jpn. J. Math., 2 (2007), 229-260.

\bibitem{levy}
B. L\'evy, {\it A numerical algorithm for L2 semi-discrete optimal transport in 3D}. ESAIM Math. Model. Numer. Anal. 49 (2015), no. 6, 1693–1715.


\bibitem{lylh}
Y.J. Liu, M.Yu, B.J. Li, Y. He, {\it  Intrinsic Manifold SLIC: A Simple and Efficient Method for Computing Content-Sensitive Superpixels.} IEEE Transactions on Pattern Analysis and Machine Intelligence 40 (2017), 653-666.

\bibitem{m}
 Q. M\'erigot, {\it  A multiscale approach to optimal transport,} Comput. Graph. Forum  30 (2011),  1583-1592.

\bibitem{okabe} 
A. Okabe, B. Boots, K. Sugihara, S.N. Chiu,  {\it Spatial tessellations: concepts and applications of Voronoi diagrams.} With a foreword by D. G. Kendall. Second edition. Wiley Series in Probability and Statistics. John Wiley \& Sons, Ltd., Chichester, 2000. 

\bibitem{pc} 
 G. Peyr\'e, L. Cohen,  Surface Segmentation Using Geodesic Centroidal Tesselation. in {\it Proceedings - 2nd International Symposium on 3D Data Processing, Visualization, and Transmission},  (2004), 995-1002. 
 
 \bibitem{Saluzzi19}
 A. Alla, M. Falcone, L. Saluzzi. {\it An efficient DP algorithm on a tree-structure for finite horizon optimal control problems.}
 SIAM Journal on Scientific Computing 41 (4), (2019) A2384--A2406.
 
 \bibitem{SethianBook}
 J.A. Sethian, {\it Level set methods and fast marching methods: evolving interfaces in computational geometry, fluid mechanics, computer vision, and materials science}. Cambridge Monographs on Applied and Computational Mathematics, 3. Cambridge University Press, Cambridge, 1999.

 \bibitem{SethianVlad}
J.A.  Sethian, A. Vladimirsky, {\it Fast methods for the Eikonal and related Hamilton–Jacobi equations on unstructured meshes. }Proceedings of the National Academy of Sciences 97 (2000), 5699-5703.


\bibitem{levy_et_altri}
S.Q. Xin, B. L\'evy, Z.  Chen, L.  Chu, Y.  Yu, C. Tu, W. Wang,
{\it Centroidal power diagrams with capacity constraints: Computation, applications, and extension,}
ACM Transactions on Graphics 35 (2016), art. no. 244.
\end{thebibliography}
\end{document}